\documentclass[a4paper,11pt]{amsart}
\usepackage[utf8]{inputenc}
\usepackage{fullpage}
\usepackage{graphicx}
\usepackage{amssymb,amsmath,latexsym}
\usepackage{tikz}
\usepackage{color}
\usepackage{feynmp,epsf}
\usepackage{sseq}
\usepackage{url}
\usepackage[colorlinks=true,linkcolor=red,citecolor=blue]{hyperref}
\theoremstyle{plain}
\newtheorem{thm}{Theorem}[section]
\newtheorem*{thm*}{\bf Theorem }
\newtheorem{lem}[thm]{Lemma}
\newtheorem{prop}[thm]{Proposition}
\newtheorem*{prop*}{\bf Proposition}
\newtheorem{cor}[thm]{Corollary}

\newtheorem{rem}[thm]{Remark}

\newtheorem{q}[thm]{Question}
\theoremstyle{definition}
\newtheorem{defn}[thm]{Definition}
\newtheorem{eg}[thm]{Example}
\numberwithin{equation}{section}
\newcommand{\I}[1]{ \mathrm{Ind} ({#1}) }

\newcommand{\setdiff}{\! \setminus \!}
\newcommand{\mi}{\! - \!}
\newcommand{\pl}{\! + \!}

\newcommand{\mb}[1]{\mathbb {#1}}

\newcommand{\image}[1]{\mathrm{im}\ \! {#1}}

\begin{document}
\title{On the homology of independence complexes}
\author{Marko Berghoff}
\address{Institut f\"ur Mathematik, Humboldt-Universit\"at zu Berlin}
\email{berghoff@math.hu-berlin.de}
\begin{abstract}
The independence complex $\mathrm{Ind}(G)$ of a graph $G$ is the simplicial complex formed by its independent sets.
This article introduces a deformation of the simplicial boundary map of $\mathrm{Ind}(G)$ that gives rise to a double complex with trivial homology. Filtering this double complex in the right direction induces a spectral sequence that converges to zero and contains on its first page the homology of the independence complexes of $G$ and various subgraphs of $G$, obtained by removing independent sets and their neighborhoods from $G$.
It is shown that this spectral sequence may be used to study the homology of $\mathrm{Ind}(G)$. 
Furthermore, a careful investigation of the sequence's first page exhibits a relation between the cardinality of maximal independent sets in $G$ and the vanishing of certain homology groups of the independence complexes of some subgraphs of $G$. This relation is shown to hold for all paths and cyclic graphs.
\end{abstract}
\maketitle

\section{Introduction}

An \textit{independent set} in a graph $G=(V,E)$ is a subset of its vertices $I \subset V$ such that no two elements in $I$ are adjacent. More generally, a subset $I\subset V$ is \textit{$r$-independent} if every connected component of the \textit{induced subgraph} $G[I]:=(I,E')$ with $E'=\{e \in E \mid e\in I\times I\}$ has at most $r$ vertices.
Since the property of being $r$-independent is closed under taking subsets, the set of all $r$-independent sets of $G$ forms a simplicial complex, the \textit{$r$-independence complex} $\mathrm{Ind}_r( G)$ of $G$; 
the vertex set of $\mathrm{Ind}_r( G)$ is $V$ and $I\subset V$ forms a simplex if and only if $I$ is $r$-independent in $G$. In the following we write $\I G$ for $\mathrm{Ind}_1( G)$. See Figure \ref{fig:g,icg,2icg} for an example.

\begin{figure}[h]\label{fig:g,icg,2icg}
$G$ \ \begin{tikzpicture}[scale=1]
  \coordinate  (v1) at (-1,0); 
   \coordinate  (v2) at (0,0);
   \coordinate (v3) at (1,0); 
   \draw (v1) -- (v2) ;
   \draw (v2) -- (v3) ;
  \fill[blue] (v1) circle (.0666cm) node[left]{$v_1$};
   \fill[red] (v2) circle (.0666cm) node[above]{$v_2$};
\fill[cyan] (v3) circle (.0666cm) node[right]{$v_3$};
   \end{tikzpicture}
   \hspace{1cm}
   $\I G$ \ \begin{tikzpicture}[scale=1]
  \coordinate  (v1) at (0,0.5); 
   \coordinate  (v2) at (0,-0.5);
   \coordinate (v3) at (0.5,0); 
   \draw (v1) -- (v2) ;
  \fill[blue] (v1) circle (.0666cm) node[left]{$v_1$};
   \fill[cyan] (v2) circle (.0666cm) node[left]{$v_3$};
\fill[red] (v3) circle (.0666cm) node[right]{$v_2$};
   \end{tikzpicture}
   \hspace{1cm}
  $\mathrm{Ind}_2(G)$ \ \begin{tikzpicture}[scale=1]
  \coordinate  (v1) at (-0.5,0); 
   \coordinate  (v2) at (0.5,0);
   \coordinate (v3) at (0,1); 
   \draw (v1) -- (v2) ;
   \draw (v2) -- (v3) ;
   \draw (v3) -- (v1);
  \fill[cyan] (v1) circle (.0666cm)node[left]{$v_3$};
   \fill[red] (v2) circle (.0666cm)node[right]{$v_2$};
\fill[blue] (v3) circle (.0666cm)node[right]{$v_1$};
   \end{tikzpicture}
   \caption{A graph, its independence complex and its 2-independence complex}
\end{figure}
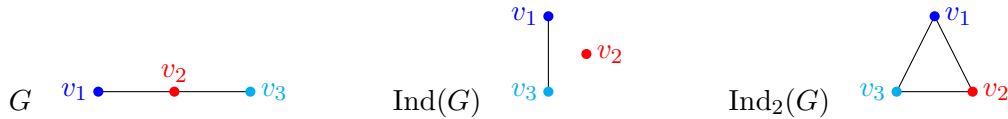

These complexes are a special instance of a great variety of simplicial complexes associated to graphs. See \cite{jonsson} for an overview. Some of these complexes are related to each other. For example, the independence complex of a graph is the \textit{matching complex} of its \textit{line graph} and the \textit{clique complex} of its \textit{dual graph}.

The topology of these types of simplicial complexes is a well-studied topic. For independence complexes, most work has been done on connectivity and homotopy-type, cf.\
\cite{kozlov,ehrenborg,szabo-tardost,engstrom2,engstrom,csorba,ADAMASZEK2012,adamaszek}, with applications, for example, to the study of graph colorings \cite{babson-kozlov,barmak}. For the case of higher independence complexes, see \cite{szabo-tardost,deshp-singh}, as well as references therein. 

The homology groups of independence complexes have also been investigated, see \cite{meshulam,jonsson2,dao-schweig}.  
Here, applications reach from statistical physics, where the Euler characteristic of $\I L$ for $L$ a periodic lattice is referred to as its \textit{Witten index} \cite{bm-sv-ne,hui-schout,adamaszek2}, to group theory, where certain local homology groups of the classical braid groups are related to the homology of certain (higher) independence complexes \cite{salvetti,paolini-salvetti}. 
\newline

The main result of this article is a computational recipe for calculating the homology groups of $\I G$ (and to some extent $\mathrm{Ind}_r(G)$, see the discussion in Remark \ref{rem:higher}). It is implied by Corollary \ref{cor:specseq} of Theorem \ref{thm:acyclic}, and strengthened by the properties established in Proposition \ref{thm:properties}. 

\begin{thm}\label{thm:1}
 Let $G$ be a finite simple graph. There exists a spectral sequence whose $E^1$-page contains a copy of the homology of $\I G$. Its other entries are given by homology groups of independence complexes of graphs $G \mi N[U]$, obtained from $G$ by deleting all vertices in $U \subset V$ together with their neighbors, $U$ running over all independent sets of $G$. The sequence collapses to $E^\infty$ which has one entry isomorphic to $\mb Z$ and all other vanishing. Moreover, the differential $d^1:E^1 \to E^1$ is explicitly given and easy to compute.
\end{thm}

This allows to study the homology of $\I G$ by monitoring the typical dramatic action associated with spectral sequences.\footnote{citing J.F.\ Adams from \cite{ug-specseq} `` \ldots \ the behavior of this spectral sequence \ldots \ is a bit like an Elizabethan drama, full of action, in which the business of each character is to kill at least one other character, so that at the end of the play one has the stage strewn with corpses and only one actor left alive (namely the one who has to speak the last few lines)." }
It effectively reduces the computation of the homology of $\I G$ to the problem of determining the homology of independence complexes of the graphs $G \mi N[U]$ and inspecting the first page(s) of the above mentioned spectral sequence. In good cases this allows to determine all homology groups of $\I G$. In general one finds at least some relations between them (and the homology of the building blocks $G\mi N[U]$). 

Moreover, the ``empirical data'' hints at a rather peculiar property of independence complexes that is satisfied by a large family of examples, including all paths and cyclic graphs; it is however not true in general.\footnote{Thanks to Dmitry Feichtner-Kozlov for pointing this out. The reader is invited to find a counter-example on her/his own. \textit{Hint:} Start with a disconnected graph...}  

\begin{thm}\label{thm:2}
Let $G$ be a path or a cyclic graph. If $G$ has no maximal independent set of cardinality $p$, then
\begin{equation*}
 \tilde H_{p-q-1}(\I {G\mi N[U]}) \cong 0
\end{equation*}
holds for all $q>0$ and all independent sets $U\subset V$ with $|U|=q$.
\end{thm}

This is proven in Theorem \ref{thm:vp}, using the observations made in Proposition \ref{thm:properties}. 
\newline

The basic idea to set up the desired spectral sequence is to consider a deformation of the simplicial boundary map $d$ of $\I G$. 
For this we model the simplicial chain complex of $\I G$ by a chain complex generated by certain decorations of the vertices of $G$. These decorations, hereafter called \textit{markings}, are given by maps $m: V \to \{  0,1\}$ with $m^{-1}(1)$ independent in $G$. The differential $d$ is given by a signed sum over all ways of removing the marking of a single vertex $v \in m^{-1}(1)$. 

We then enhance this picture by introducing a second type of marking, i.e.\ we consider now maps $m:V \to \{ 0,1,2 \}$ with $m^{-1}(\{ 1,2 \})$ independent. This allows to define a second differential $\delta$ that changes the first into the second type. The two differentials anticommute, so that we can form a double complex of markings on $G$, graded by the number of markings of the first and second type, called \textit{1-} and \textit{2-markings}, respectively. Setting $D:=d + \delta$ defines a differential on the total complex
\begin{equation*}
T(G):=\bigoplus_n T_n(G) ,\ T_n(G):=\bigoplus_{i-j=n}T_{i,j}(G),
\end{equation*}
where $T_{i,j}(G)$ is the free abelian group generated by markings $m:V \to \{ 0,1,2 \}$ with $i$ marked vertices in total and $j$ 2-marked vertices.

It turns out that this total complex $(T(G),D)$ is acyclic. 
Filtering it by the number of 2-markings induces a spectral sequence with first page
\begin{equation*}
E^1_{p,q}= H_p(T_{\bullet,q}(G), d).
\end{equation*}
Thus, the row $q=0$ contains the homology of $\I G$, while the entries with $q>0$ are identified with the homology groups of graphs $G\mi N[U]$ for $U\subset V$ independent and $|U|=q$. Since the spectral sequence converges to zero, this allows to apply standard techniques from homological algebra to study the homology of $\I G$.
\newline

The whole construction is based on the article \cite{ksvs} where two similar complexes (of edge- and cycle-markings) were used to encode consistency conditions in the perturbative quantization of non-abelian gauge theories. 

The masters thesis \cite{knispel} by Knispel studied the cohomology of these complexes in detail; he showed that every variant of marking can be pulled back to the case of marking vertices in an associated simple graph $G$, allowing to compute the cohomology of all such complexes at once (a streamlined version of this construction, using the above introduced spectral sequence in a slightly different disguise, can be found in \cite{mb-ak}). He then continued to study the nontrivial part $d$ of the differential $D=d+\delta$, relating it to the notion of independent sets and cliques in a graph $G$. 

In this article we show that this relation can in fact be pushed much further. Firstly, the map $d$ really \emph{is} the boundary map of the independence complex of $G$, and secondly, the total complex $(T(G),D)$ may be used to study its homology.
\newline

The exposition is organized as follows. In Section \ref{sec:ismark} we define the notion of markings to model independent sets in a graph $G$. We then introduce two differentials $d$ and $\delta$ to set up a double complex $(T(G),D)$ that contains a copy of the simplicial chain complex of $\I G$. 

The next two sections recite the results of \cite{mb-ak} (Sections 3.1 and 3.2 therein). In Section \ref{sec:delta} the vertical differential $\delta$ of $(T(G),D)$ is studied and its homology is shown to be trivial, except in bidegree $(0,0)$ where it is isomorphic to $\mb Z$. We use this in Section \ref{sec:doublecomplex} to compute the homology of the total complex $(T(G),D)$, showing that it is acyclic as well.

Section \ref{sec:homology} contains the heart of this article. It introduces the spectral sequence of Theorem \ref{thm:1} that allows to study the homology of $\I G$. We proceed then by investigating its most important properties. The section finishes with a discussion of the property mentioned in Theorem \ref{thm:2}, the relation between the nonexistence of maximal independent sets and the vanishing of certain homology groups of independence complexes of subgraphs of $G$. We proof this for the case of paths and cyclic graphs. The extension of this statement to other families of graph (and higher independence complexes) is left as an open problem.

In Section \ref{sec:eg} we look at some elaborated examples.

\section{Independent sets and markings}\label{sec:ismark}
Let $G=(V,E)$ be a finite, simple graph. 
We start by introducing a model for independent sets $I \subset V$ of $G$. For this we simply label the vertices in $I$, and call such a labeling a \textit{marking} of $G$. The raison d’être is that this point of view allows to 
\begin{enumerate}
 \item model the simplicial boundary map of $\I G$ as a map on $G$ that removes labels on vertices,
 \item introduce a second kind of label which gives rise to a deformation of the simplicial boundary map of $\I G$.
\end{enumerate}

In what follows everything will depend on the chosen graph $G$, but whenever there is no risk of confusion, this dependence is dropped from notation. 

Throughout this paper $H$ and $\tilde H$ denote homology and reduced homology, respectively, with integer coefficients.

\begin{defn}\label{defn:marking}
Let $G$ be a graph. A \textit{marking} of $G$ is a map $m:V \to \{0,1,2\}$ such that $V_m:=m^{-1}(\{1,2\})$ is an independent set in $G$. For $i=1,2$ we refer to the elements of $V_i:=m^{-1}(i)$ as \textit{$i$-marked} and to the elements of $V_0:=m^{-1}(0)$ as \textit{unmarked}.
\end{defn}

\begin{defn}\label{defn:T}
Choose an order on $V$ such that $V=\{v_1,\ldots, v_n\}$ with $v_i<v_j$ if and only if $i<j$.

 Let $T_{i,j}=T_{i,j}(G)$ be the free abelian group generated by all markings of $G$ with $i$ marked and $j$ 2-marked vertices,
\begin{equation*}
T_{i,j}:= \mb Z \big\langle m:V \to \{0,1,2\} \mid |V_m|=i, |V_2|=j  \big\rangle.
\end{equation*}

Define linear maps $d: T_{i,j}\to T_{i-1,j}$ and $\delta: T_{i,j}\to T_{i,j+
1}$ by
\begin{align*}
d m & :=\sum_{v \in V_1} (-1)^{\# \{ w \in V_1 \mid w<v \} } m_{v\mapsto 0},
\\
\delta m & :=\sum_{v \in V_1} (-1)^{\# \{ w \in V_1 \mid w<v \} } m_{v\mapsto 2},
\end{align*}
where 
\begin{equation*}
m_{v\mapsto i}(x):=
\begin{cases} m(x) & \text{ if } x\neq v,\\
 i & \text{ if } x=v,
\end{cases}
\end{equation*}
 changes the marking $m$ by relabeling the vertex $v$ with $i$. 
\end{defn}

\begin{eg}\label{eg:firstexample}
Let $G=P_3=
\raisebox{0.023cm}{\begin{tikzpicture}[scale=1]
  \coordinate  (v1) at (-1,0); 
   \coordinate  (v2) at (0,0);
   \coordinate (v3) at (1,0); 
   \draw (v1) -- (v2) ;
   \draw (v2) -- (v3) ;
    \filldraw[fill=black] (v1) circle (.0666);
  \filldraw[fill=black] (v2) circle (.0666) ;
  \filldraw[fill=black] (v3) circle (.0666);
   \end{tikzpicture}}$
with $V=\{v_1,v_2,v_3\}$ ordered from left to right. Let us denote 1-marked and 2-marked vertices by orange and white filled circles. For instance, the marking 
\begin{equation*}
m: v_1 \longmapsto 1, v_2 \longmapsto 0,v_3 \longmapsto 1
\end{equation*}
is represented by 
\begin{equation*}
m= 
\begin{tikzpicture}[scale=1]
  \coordinate  (v1) at (-1,0); 
   \coordinate  (v2) at (0,0);
   \coordinate (v3) at (1,0); 
   \draw (v1) -- (v2);
   \draw (v2) -- (v3);
   \filldraw[fill=orange] (v1) circle (0.1);
  \filldraw[fill=black] (v2) circle (.0666) ;
  \filldraw[fill=orange] (v3) circle (.1);
   \end{tikzpicture}.
   \end{equation*}
   
   Computing the differentials gives
\begin{align*}
 d \
 \begin{tikzpicture}[scale=1]
  \coordinate  (v1) at (-1,0); 
   \coordinate  (v2) at (0,0);
   \coordinate (v3) at (1,0); 
   \draw (v1) -- (v2);
   \draw (v2) -- (v3);
   \filldraw[fill=orange] (v1) circle (0.1);
  \filldraw[fill=black] (v2) circle (.0666) ;
  \filldraw[fill=orange] (v3) circle (.1);
   \end{tikzpicture}
    &= 
\begin{tikzpicture}[scale=1]
  \coordinate  (v1) at (-1,0); 
   \coordinate  (v2) at (0,0);
   \coordinate (v3) at (1,0); 
   \draw (v1) -- (v2) ;
   \draw (v2) -- (v3) ;
  \filldraw[fill=black] (v1) circle (.0666);
  \filldraw[fill=black] (v2) circle (.0666) ;
  \filldraw[fill=orange] (v3) circle (.1);
   \end{tikzpicture}
   -
    \begin{tikzpicture}[scale=1]
  \coordinate  (v1) at (-1,0); 
   \coordinate  (v2) at (0,0);
   \coordinate (v3) at (1,0);  
   \draw (v1) -- (v2) ;
   \draw (v2) -- (v3) ;
   \filldraw[fill=orange] (v1) circle (.1);
  \filldraw[fill=black] (v2) circle (.0666) ;
  \filldraw[fill=black] (v3) circle (.0666);
   \end{tikzpicture} 
   , \\
    \delta \ 
    \begin{tikzpicture}[scale=1]
  \coordinate  (v1) at (-1,0); 
   \coordinate  (v2) at (0,0);
   \coordinate (v3) at (1,0); 
   \draw (v1) -- (v2);
   \draw (v2) -- (v3);
   \filldraw[fill=orange] (v1) circle (0.1);
  \filldraw[fill=black] (v2) circle (.0666) ;
  \filldraw[fill=orange] (v3) circle (.1);
   \end{tikzpicture}
    &  = 
   \begin{tikzpicture}[scale=1]
  \coordinate  (v1) at (-1,0); 
   \coordinate  (v2) at (0,0);
   \coordinate (v3) at (1,0); 
    \draw (v1) -- (v2) ;
   \draw (v2) -- (v3) ;   
   \filldraw[fill=white] (v1) circle (.1);
  \filldraw[fill=black] (v2) circle (.0666) ;
  \filldraw[fill=orange] (v3) circle (.1);
   \end{tikzpicture}
   -
    \begin{tikzpicture}[scale=1]
  \coordinate  (v1) at (-1,0); 
   \coordinate  (v2) at (0,0);
   \coordinate (v3) at (1,0); 
    \draw (v1) -- (v2) ;
   \draw (v2) -- (v3) ;   
  \filldraw[fill=orange] (v1) circle (.1);
  \filldraw[fill=black] (v2) circle (.0666) ;
  \filldraw[fill=white] (v3) circle (.1);
   \end{tikzpicture}, 
   \\
    \delta d  \
\begin{tikzpicture}[scale=1]
  \coordinate  (v1) at (-1,0); 
   \coordinate  (v2) at (0,0);
   \coordinate (v3) at (1,0); 
   \draw (v1) -- (v2);
   \draw (v2) -- (v3);
   \filldraw[fill=orange] (v1) circle (0.1);
  \filldraw[fill=black] (v2) circle (.0666) ;
  \filldraw[fill=orange] (v3) circle (.1);
   \end{tikzpicture}    
    & =
   \begin{tikzpicture}[scale=1]
  \coordinate  (v1) at (-1,0); 
   \coordinate  (v2) at (0,0);
   \coordinate (v3) at (1,0); 
    \draw (v1) -- (v2) ;
   \draw (v2) -- (v3) ;   
   \filldraw[fill=black] (v1) circle (.0666);
  \filldraw[fill=black] (v2) circle (.0666) ;
  \filldraw[fill=white] (v3) circle (.1);
   \end{tikzpicture}
   -
   \begin{tikzpicture}[scale=1]
  \coordinate  (v1) at (-1,0); 
   \coordinate  (v2) at (0,0);
   \coordinate (v3) at (1,0); 
    \draw (v1) -- (v2) ;
   \draw (v2) -- (v3) ;   
   \filldraw[fill=white] (v1) circle (.1);
  \filldraw[fill=black] (v2) circle (.0666) ;
  \filldraw[fill=black] (v3) circle (.0666);
   \end{tikzpicture}
   = - d \delta \ 
   \begin{tikzpicture}[scale=1]
  \coordinate  (v1) at (-1,0); 
   \coordinate  (v2) at (0,0);
   \coordinate (v3) at (1,0); 
   \draw (v1) -- (v2);
   \draw (v2) -- (v3);
   \filldraw[fill=orange] (v1) circle (0.1);
  \filldraw[fill=black] (v2) circle (.0666) ;
  \filldraw[fill=orange] (v3) circle (.1);
   \end{tikzpicture}.
\end{align*}

If on the other hand $m= 
\begin{tikzpicture}[scale=1]
  \coordinate  (v1) at (-1,0); 
   \coordinate  (v2) at (0,0);
   \coordinate (v3) at (1,0); 
   \draw (v1) -- (v2) ;
   \draw (v2) -- (v3) ;
  \filldraw[fill=black] (v1) circle (.0666);
  \filldraw[fill=white] (v2) circle (.1) ;
  \filldraw[fill=black] (v3) circle (.0666);
   \end{tikzpicture}$, then $d m = \delta m = 0$.
\end{eg}

Our goal is to define a deformation of $d$ using the map $\delta$, that is we want $d+\delta$ to be a differential as well.

\begin{prop}\label{prop:differential}
$d^2=\delta^2=0$ and $d\delta + \delta d=0$.
\end{prop}

\begin{proof}
The first statement follows by a standard computation,
\begin{align*}  
d dm & = d \sum_{v \in V_1 }(-1)^{\# \{ w \in V_1 \mid w<v \} } m_{v\mapsto 0}  \\
  &=  \sum_{v \in V_1 }(-1)^{\# \{ w \in V_1 \mid w<v \} }\sum_{v' \in V_1 \setdiff \{v\} }(-1)^{ \# \{ w' \in V_1 \setdiff \{v\} \mid w' < v' \}} m_{v,v' \mapsto 0} \\
  &=  \sum_{v \in V_1 }\sum_{v' \in V_1 \setdiff \{v\}} (-1)^{ \# \{ w \in V_1 \mid w < v\} + \# \{ w' \in V_1 \setdiff \{v\} \mid w' < v'\}} m_{v,v' \mapsto 0} \\
  & = \sum_{v,v' \in V_1,v'<v }(-1)^{\# \{ u \in V_1 \mid v'< u < v\}+1 }m_{v,v' \mapsto 0} \\
  & \quad + \sum_{v,v' \in V_1, v'>v } (-1)^{\# \{ u \in V_1 \mid v< u < v' \}}   m_{v,v' \mapsto 0}= 0,
 \end{align*}
  and similarly for $\delta$.
  
The same argument shows that $d$ and $\delta$ anticommute,
\begin{align*}
  d \delta m & = d \sum_{v \in V_1 }(-1)^{\# \{ w \in V_1 \mid w < v\} }  m_{v \mapsto 2} \\
  &= \sum_{v \in V_1} \sum_{ v' \in V_1\setdiff \{v\} }(-1)^{ \# \{ w \in V_1 \mid w < v\}+ \# \{ w' \in V_1 \setdiff \{v\}  \mid w' < v'\}} m_{v \mapsto 2,v' \mapsto 0}  \\
  &=\sum_{v,v' \in V_1, v'<v} (-1)^{ \# \{ u \in V_1 \mid  v' < u < v \} +1} m_{v\mapsto 2,v' \mapsto 0} \\
  & \quad + \sum_{v,v' \in V_1,v'>v }(-1)^{ \# \{ u \in V_1 \mid v < u < v'  \} } m_{v \mapsto 2,v' \mapsto 0} \\
   &= (-1) \Big( \sum_{v,v' \in V_1, v'<v} (-1)^{ \# \{ u \in V_1 \mid  v' < u < v \} } m_{v\mapsto 2,v' \mapsto 0} \\
  & \quad + \sum_{v,v' \in V_1,v'>v }(-1)^{ \# \{ u \in V_1 \mid v < u < v'  \} +1} m_{v \mapsto 2,v' \mapsto 0} \Big) \\
    &= (-1) \Big(  \sum_{v,v' \in V_1,v < v' }(-1)^{ \# \{ u \in V_1 \mid v < u < v'  \} +1}  m_{v\mapsto 2,v' \mapsto 0} \\
  & \quad + \sum_{v \in V_1, v > v'} (-1)^{ \# \{ u \in V_1 \mid  v' < u < v \} } m_{v \mapsto 2,v' \mapsto 0} \Big) \\
  & = -\delta d m.
   \end{align*} 
  \end{proof}

Note that the complex $(T_{\bullet,0},d)$ models the simplicial chain complex of $\mathrm{Ind}(G)$. More precisely, for any choice of order on $V$ and orientation of $\mathrm{Ind}(G)$ there exists a unique isomorphism of chain complexes
\begin{equation}\label{eq:zerocomplex}
(T_{\bullet,0},d) \cong (C_{\bullet-1}(\mathrm{Ind}(G)), \partial)
\end{equation}
where $(C_\bullet(\mathrm{Ind}(G)), \partial)$ denotes the augmented simplicial chain complex of $\mathrm{Ind}(G)$. On the level of chain groups this isomorphism is given by simply mapping every independent set $I$ in $G$ to the marking $m_I$ that marks the vertices in $I$ by 1 and every other vertex by 0. Since an orientation of $\mathrm{Ind}(G)$ is the same as a linear order on its vertex set, which is equal to $V$, this correspondence defines a chain map. Thus, $H_n(T_{\bullet,0},d)$ is isomorphic to the reduced simplicial homology $\tilde H_{n-1}(\mathrm{Ind}(G))$ of the independence complex of $G$.
\newline

What about the complexes with 2-marked vertices, i.e.\ the case $j\neq 0$? In this case we can relate the complex $(T_{\bullet,j},d)$ to independence complexes of graphs obtained from $G$ by removing $j$ vertices together with their neighborhoods.

\begin{defn}
The \textit{neighborhood} of a vertex $v$ in a graph $G$ is $N(v):= \{ w \in V \mid  \{v,w\} \in E  \}$, the set of vertices of $G$ adjacent to $v$. The \textit{closed neighborhood} of $v$ is $N[v]:= \{v\} \cup N(v)$. Likewise, for a subset $U \subset V$ we define 
\begin{equation*}
N(U):= \bigcup_{v \in U} N(v) \ \text{ and } \ N[U]:= \bigcup_{v \in U} N[v]
\end{equation*}

For any subset $W\subset V$ we write $G - W$ for  the graph obtained from $G$ by deleting all elements in $W$, i.e.\ for the graph $G- W:=(V',E')$ with
\begin{equation*}
V'=V \setdiff W, \quad E'= \{ e=\{x,y\} \in E \mid x,y \in V'  \},
\end{equation*}
also known as the \textit{induced subgraph} $G[V-W]$.
\end{defn}

\begin{prop}\label{prop:dcomplex}
For each $j\geq 0$ there is an isomorphism of chain complexes, 
\begin{equation}\label{eq:dcomplex}
(T_{\bullet, j}(G),d) \cong \bigoplus_{ \substack{ U\subset V \text{ independent}  \\ |U|=j }  } ( C_{\bullet-1}(G \mi N[U]),\partial ),
\end{equation}
$(C_{\bullet-1}(G\mi N[U]),\partial)$ denoting the augmented (and degree shifted) simplicial chain complex of $\I {G\mi N[U]}$ with the convention 
\begin{equation*}
(C_{\bullet}(\emptyset),\partial) := 0 \overset{\partial_0}{\longrightarrow} \mb Z \overset{\partial_{-1}}{\longrightarrow} 0.
\end{equation*}
\end{prop}

\begin{proof}
The case $j=0$ has been discussed above, so let $j>0$. By definition every set of 2-marked vertices forms an independent set in $G$. Conversely, every independent set of size $j$ can be modeled by an appropriate 2-marking. Since $d$ acts only on 1-marked vertices, the complex $(T_{\bullet,j},d)$ splits into a direct sum of complexes with one summand for each independent/2-marked set of size $j$. 

Given such an independent set $U$, the remaining vertices that can be marked are precisely the non-neighbors of vertices in $U$, i.e.\ the vertices of the graph $G\mi N[U]$. If $G\mi N[U]$ is not empty, then \eqref{eq:dcomplex} follows from the interpretation \eqref{eq:zerocomplex}.

In the special case that $U$ is a maximal independent set there is no possibility to put any 1-markings on the remaining vertices, so $G\mi N[U]=\emptyset$. The corresponding chain complex has only one nontrivial chain group in degree zero, generated by a single element, the marking that marks every vertex in $U$ by 2.
\end{proof}

In terms of chain complexes of markings $\left(T(\cdot),d\right)$ we may rephrase the previous proposition. It states that 
\begin{equation} \label{eq:dcomplex2}
(T_{\bullet, j}(G),d) \cong \bigoplus_{ \substack{ U\subset V \text{ independent}  \\ |U|=j }  } ( T_{\bullet-j,0}(G\mi N[U]),d ).
\end{equation}

This leads to two important observations. 

Firstly, even in the presence of 2-marked vertices the complexes $(T_{\bullet,j},d)$ can be interpreted as chain complexes of independence complexes of graphs (more precisely, of subgraphs of $G$). 

Secondly, as the number $j$ of 2-markings increases, the complexes $(T_{\bullet, j},d)$ split into simpler and simpler building blocks. 

We will see that these two observations -- together with the simplicity of the second differential $\delta$ -- allow us to set up a spectral sequence to study the homology of $(T_{\bullet, 0},d)$, and thus $H_\bullet(\mathrm{Ind}(G))$.

\begin{rem}\label{rem:higher}
The whole construction outlined in this paper works also in the case of higher independence complexes (as well as for more general markings where a given set of subgraphs is allowed to be marked, cf.\ \cite{mb-ak}).
One simply requires in Definition \ref{defn:marking} the set $V_m$ of marked elements to be $r$-independent. Then everything works exactly as presented here for the case $r=1$, except for one crucial difference. The splitting of $(T_{\bullet, j},d)$ in \eqref{eq:dcomplex} or \eqref{eq:dcomplex2} becomes more complicated: The direct sum runs now over $r$-independent sets in $G$ and the appropriate replacements of the graphs $G\mi N[U]$ are not necessarily subgraphs of $G$ anymore due to the varying cardinality of $r$-independent sets. Only for $r=2$ a similar looking formula can be recovered where the summands in \eqref{eq:dcomplex2} have to be replaced by 1-independence complexes of appropriately associated graphs. See Example \ref{eg:2ind} and \ref{eg:2ind2}.
\end{rem}

\section{The second differential $\delta$}\label{sec:delta}

While $d$ models the boundary map of independence complexes, the map $\delta: T_{i,j} \to T_{i,j+1}$ acts by relabeling already marked vertices. Thus, it is effectively independent of the topology of $G$. 
However, this differential may also be interpreted as the (co-)boundary map of a simplicial complex, albeit of a very simple one, the standard simplex $\Delta^{i-1}$ on $i$ vertices.

\begin{rem}
Note that $\delta$ has bidegree $(0,1)$,  going in the ``wrong" direction. Nevertheless, we will use homological terminology for both maps, $d$ and $\delta$. This avoids awkwardly changing between homology and cohomology. For the purists this choice of convention may be justified by flipping the sign in the second part of the grading of $T_{i,j}$, i.e.\ by defining 
\begin{equation*}
T_{i,j}:= \mb Z \big\langle m:V \to \{0,1,2\} \mid |V_m|=i, |V_2|=-j  \big\rangle.
\end{equation*}
\end{rem}

\begin{prop}\label{prop:deltahomtrivial}
All homology groups $H_n( T_{i,\bullet} , \delta )$ are trivial unless $i=0$ and $n=0$. In this case $H_0( T_{0,\bullet} , \delta )\cong \mb Z$, generated by the trivial marking 
\begin{equation}
m_0:V \longrightarrow \{0,1,2\}, \ m_0(v) = 0   \text{ for all } v\in V.
\end{equation}
\end{prop}

The proof of this proposition is not needed in the sequel and may be omitted on a first read; the following example should give an intuitive idea why the statement holds. 

\begin{eg} 
Let $G=P_3$ as in Example \ref{eg:firstexample}. For $i=2$ we have 
   \begin{align*}
   \delta \
\begin{tikzpicture}[scale=1]
  \coordinate  (v1) at (-1,0); 
   \coordinate  (v2) at (0,0);
   \coordinate (v3) at (1,0); 
   \draw (v1) -- (v2);
   \draw (v2) -- (v3);
   \filldraw[fill=orange] (v1) circle (0.1);
  \filldraw[fill=black] (v2) circle (.0666) ;
  \filldraw[fill=orange] (v3) circle (.1);
   \end{tikzpicture} 
   & = 
   \begin{tikzpicture}[scale=1]
  \coordinate  (v1) at (-1,0); 
   \coordinate  (v2) at (0,0);
   \coordinate (v3) at (1,0); 
   \draw (v1) -- (v2);
   \draw (v2) -- (v3);
   \filldraw[fill=white] (v1) circle (0.1);
  \filldraw[fill=black] (v2) circle (.0666) ;
  \filldraw[fill=orange] (v3) circle (.1);
   \end{tikzpicture}
   -
   \begin{tikzpicture}[scale=1]
  \coordinate  (v1) at (-1,0); 
   \coordinate  (v2) at (0,0);
   \coordinate (v3) at (1,0); 
   \draw (v1) -- (v2);
   \draw (v2) -- (v3);
   \filldraw[fill=orange] (v1) circle (0.1);
  \filldraw[fill=black] (v2) circle (.0666) ;
  \filldraw[fill=white] (v3) circle (.1);
   \end{tikzpicture}
 ,   \\
  \delta \ 
  \begin{tikzpicture}[scale=1]
  \coordinate  (v1) at (-1,0); 
   \coordinate  (v2) at (0,0);
   \coordinate (v3) at (1,0); 
   \draw (v1) -- (v2);
   \draw (v2) -- (v3);
   \filldraw[fill=white] (v1) circle (0.1);
  \filldraw[fill=black] (v2) circle (.0666) ;
  \filldraw[fill=orange] (v3) circle (.1);
   \end{tikzpicture}
   & = 
   \begin{tikzpicture}[scale=1]
  \coordinate  (v1) at (-1,0); 
   \coordinate  (v2) at (0,0);
   \coordinate (v3) at (1,0); 
   \draw (v1) -- (v2);
   \draw (v2) -- (v3);
   \filldraw[fill=white] (v1) circle (0.1);
  \filldraw[fill=black] (v2) circle (.0666) ;
  \filldraw[fill=white] (v3) circle (.1);
   \end{tikzpicture}
   = 
   \delta \
   \begin{tikzpicture}[scale=1]
  \coordinate  (v1) at (-1,0); 
   \coordinate  (v2) at (0,0);
   \coordinate (v3) at (1,0); 
   \draw (v1) -- (v2);
   \draw (v2) -- (v3);
   \filldraw[fill=orange] (v1) circle (0.1);
  \filldraw[fill=black] (v2) circle (.0666) ;
  \filldraw[fill=white] (v3) circle (.1);
   \end{tikzpicture}
   , \\
   \delta \
   \begin{tikzpicture}[scale=1]
  \coordinate  (v1) at (-1,0); 
   \coordinate  (v2) at (0,0);
   \coordinate (v3) at (1,0); 
   \draw (v1) -- (v2);
   \draw (v2) -- (v3);
   \filldraw[fill=white] (v1) circle (0.1);
  \filldraw[fill=black] (v2) circle (.0666) ;
  \filldraw[fill=white] (v3) circle (.1);
   \end{tikzpicture} 
   & = 0.
   \end{align*}
\end{eg}

A formal proof of Proposition \ref{prop:deltahomtrivial} relies on the following two lemmata.

\begin{lem}
For $k>1$ let $(C_\bullet(k),\partial)$ denote the augmented and degree shifted simplicial chain complex of the standard simplex on $k$ vertices, $(C_\bullet(k),\partial):=(C_{\bullet-1}(\Delta^{k-1}),\partial)$.
Then for $i>0$
\begin{equation*}
 (T_{i,\bullet},\delta) \cong \bigoplus_{ \substack{U \subset V \text{independent} \\ |U | =i} } (C_\bullet(i),\partial).
\end{equation*}
\end{lem}

\begin{proof}
Fix $i>0$. The map $\delta$ neither creates any new marked vertices nor does it see adjacency relations in $G$; it simply operates on the set of marked vertices, regardless of their distribution in $G$. Therefore, the complex $(T_{i,\bullet},\delta)$ splits into a direct sum
\begin{equation*}
(T_{i,\bullet},\delta) \cong \bigoplus_{ \substack{U \subset V \text{independent} \\ |U | =i} } (T^{f}(U)_\bullet,\delta),
\end{equation*}
where 
\begin{equation*}
T^{f}(U):=\mb Z \langle m:U\to \{1,2\} \rangle
\end{equation*}
denotes the free abelian group generated by all ``full" markings of the graph $U$ on $i=|U|$ disjoint vertices, graded by the number of 1-marked elements. 

Identifying the vertices in $U$ with the vertices of $\Delta^{i-1}$, there is a unique orientation preserving bijection between the elements $m \in T^{f}(U)$ and the simplices in $\Delta^{i-1}$, sending $m$ to the (oriented) simplex $m^{-1}(1) \subset U$. This is clearly a chain map (after shifting the degree by one), so that
\begin{equation*}
(T^{f}(U)_\bullet,\delta) \cong (C_{\bullet-1}(\Delta^{i-1}),\partial)
\end{equation*}
and the claim follows.
\end{proof}

\begin{lem}
For each $i>0$ and all $n\in \mb N$ the groups $H_n ( C_\bullet(i),\partial \big)$ are trivial.
\end{lem}

\begin{proof}
A simplex is contractible, so its reduced homology vanishes.
\end{proof}

\begin{proof}[Proof of Proposition \ref{prop:deltahomtrivial}]
Combining the two lemmata shows the first assertion in the proposition, the second one follows by direct computation: Clearly, $m_0 \in \ker \delta$, and since $\delta$ keeps the number of marked elements constant, $m_0$ cannot be an element of $\mathrm{im} \ \! \delta$.
\end{proof} 

Our next task is to study the differential $d+\delta$, viewed as a deformation of $d$, and to use it to extract information about the differential $d$, especially when restricted to the subcomplex of markings with no 2-marked vertices.

\section{The total complex and its homology}\label{sec:doublecomplex}

To study the homology of $(T_{\bullet, 0},d)$ we first consider the total complex associated to $T_{i,j}$ with $d$ as horizontal and $\delta$ as vertical differential.  
For this let
\begin{equation*}
  T:= \bigoplus_n T_n, \quad T_n:= \bigoplus_{i-j=n} T_{i,j}.
\end{equation*}

Note that the total grading is given by the number of 1-marked vertices. Define a differential on $T$ by
\begin{equation*}
D_n: T_n \longrightarrow T_{n-1}, \ D_n:=d+\delta .
\end{equation*}

Proposition \ref{prop:differential} implies that $(T,D)$ is a chain complex. Its homology is given by

\begin{thm}\label{thm:acyclic}
The complex $(T,D)$ is acyclic,
\begin{equation*}
H_n(T,D)\cong\begin{cases}
\mb Z & n=0, \\
0 & n \neq 0.
\end{cases}
\end{equation*}
\end{thm}

\begin{proof}
Consider an ascending exhaustive filtration 
\begin{equation*}
 0=F_{-1}T \subset \ldots \subset F_{p-1}T \subset F_pT \subset \ldots \subset T, 
\end{equation*}
defined by
\begin{equation}\label{eq:filtration}
F_pT_n:= \bigoplus_{ \substack{ i-j=n, \\ i\leq p}} T_{i,j}.
\end{equation}

It induces an associated spectral sequence which starts with
\begin{align*}
& E^0_{p,q}:= F_p T_{p-q} / F_{p-1} T_{p-q} = \bigoplus_{ \substack{ i-j=p-q, \\ i = p}} T_{i,j} = T_{p,q} ,
\\
& d^0_{p,q} :  E^0_{p,q} \longrightarrow E^0_{p,q+1} = \delta : T_{p,q} \longrightarrow T_{p,q+1}.
\end{align*}

On its first page we have $E^1_{p,q}=H_q(T_{p,\bullet},\delta)$ which by Proposition \ref{prop:deltahomtrivial} vanishes for $(p,q)\neq (0,0)$, while for $p=q=0$ we have $H_0(T_{0,\bullet},\delta) \cong \mb Z$. 

It follows that all differentials on the next page 
\begin{equation*}
d^1_{p,q}: E^1_{p,q} \longrightarrow E^1_{p-1,q}
\end{equation*} 
are trivial and the sequence collapses with $E^\infty=E^1$. 

By standard spectral sequence arguments\footnote{We follow the conventions in \cite{gm-homalg}. For nice introductions see \cite{chow-specseq} as well as \cite{ramos-specseq}, and \cite{ug-specseq} for a concise treatment of the subject.} we can reconstruct from $E^\infty$ the (associated graded pieces of the) homology of $(T,D)$. In the present case this is simple, 
\begin{equation*}
H_0(T,D)\cong E^1_{0,0} \cong \mb Z, \quad H_n(T,D) \cong 0 \text{ for all }n>0.
\end{equation*} 
\end{proof}

In the next section we will study a spectral sequence associated to the other canonical filtration of $T$, obtained by filtering in the horizontal direction. This spectral sequence converges to the same limit. Furthermore, its first page will be populated by the homology groups of $\mathrm{Ind}(G)$ and of the independence complexes of the graphs $G\mi N[U]$ for $U\subset V$ independent.

\section{The homology of $\mathrm{Ind}(G)$}\label{sec:homology}

We now turn our attention to the homology groups $H_n(T_{\bullet,0},d)\cong \tilde{H}_{n-1}(\mathrm{Ind}(G))$. The proof of Theorem \ref{thm:acyclic} implies the following

\begin{cor}\label{cor:specseq}
There is a spectral sequence converging to $H_n(T,D)$ with its first page containing a copy of the (reduced) homology of the independence complex of $G$.
\end{cor}

\begin{proof}
Consider the spectral sequence associated to a filtration of $T$, ``orthogonal" to the one in \eqref{eq:filtration},
\begin{equation*}
T=F_{0}T \supset   \ldots \supset F_pT \supset F_{p+1}T \supset \ldots \supset 0
\end{equation*}
with
\begin{equation*}
F_pT_n := \bigoplus_{ \substack{ i-j=n \\ j\geq p } } T_{i,j}.
\end{equation*}

Since there are only finitely many nonvanishing $T_{i,j}$, the associated spectral sequence converges to the same limit as the one in the proof of Theorem \ref{thm:acyclic} (this is a standard fact; see, for instance, Proposition 3.5.1 in \cite{gm-homalg}). Its starting page $E^0$ is given by
\begin{align*}
& E^0_{p,q} = F_pT_{p-q} / F_{p+1}T_{p-q}= \bigoplus_{ \substack{ i-j=p-q, \\ j = p} } T_{i,j} = T_{2p-q,p}, \\
& d^0_{p,q}: E^0_{p,q} \longrightarrow E^0_{p,q+1}=d: T_{2p-q,p} \longrightarrow T_{2p-q-1,p}.
\end{align*}

Here it is important to note that these unusual index shifts arise because $d$ is of bidegree $(-1,0)$ while $\delta$ is of bidegree $(0,1)$. If we reshuffle $p$ and $q$ (or tilt and stretch our sheet of paper), then we may return to the conventional picture with
\begin{equation*}
 E^0_{p,q} = T_{p,q} \ \text{ and } \ d^0_{p,q}: T_{p,q} \to T_{p-1,q}.
\end{equation*}

With these conventions we find on the next page  
\begin{equation*}
E^1_{p,q}= H_p(T_{\bullet,q},d),
\end{equation*} 
and the maps $d^1_{p,q}:H_p(T_{\bullet,q},d)\to H_p(T_{\bullet,q-1},d)$ are induced by $\delta$,
\begin{equation*}
d^1[x] = [\delta x]. 
\end{equation*} 

Proposition \ref{prop:dcomplex} and Equation \eqref{eq:dcomplex2} identify the $q=0$ column with $\tilde{H}_{\bullet}(\mathrm{Ind}(G))$. The other columns are populated by direct sums of the homology groups of $\I {G\mi N[U]}$ for $U \subset V$ independent with $|U|=q$.
\end{proof}

The usefulness of this corollary lies in the simplicity of the spectral sequence's limit $E^\infty$. If we know the homology of the complexes $(T_{\bullet,j},d)$ for $j>0$, we can deduce information about $H_\bullet(\mathrm{Ind}(G))$ from studying the $E^1$ (or $E^2$-) page of this spectral sequence. Since we know that it eventually collapses, every entry except for a single copy of $\mb Z$ must disappear at some stage. 

In more dramatic words, as the spectral sequences progresses further, all but one entries of $E^1$ are eventually paired together, both partners doomed to killing each other. 
There will be only one lucky survivor, a representative from the group of maximum independent sets of $G$.

\begin{eg}
Let us look at a very simple example, $G=K_n$ the complete graph on $n$ vertices (more sophisticated examples can be found in Section \ref{sec:eg}). The $E^1$ page of the spectral sequence from Corollary \ref{cor:specseq} has two nontrivial rows, $E^1_{p,0}=H_p(T_{\bullet,0},d)$ containing the homology of $\I {K_n}$ and 
\begin{equation*}
E^1_{p,1}=H_p(T_{\bullet,1},d)  \overset{ \eqref{eq:dcomplex2}}{ \cong } \bigoplus_{ \substack{ U\subset V \text{ independent}  \\ |U|=1 }  } H_{p-1}( T_{\bullet,0}(G\mi N[U]),d ).
\end{equation*}
Here $G-N[U]= \emptyset$ because every vertex of $K_n$ is already a maximal (and maximum) independent set. By Proposition \ref{prop:dcomplex} the row $E^1_{p,1}$ is thus given by $E^1_{1,1}  \cong  \mb Z^n$ and $E^1_{p,1}  \cong  0$ for $p\neq 1$: 
\begin{equation*}
\begin{tikzpicture}
\draw (0,0) node {0} (1,0) node {$E^1_{1,0}$} (1,1) node {$\mb Z^n$} (2,0) node {0} (2,1) node {0}  (2,2) node {0} ;
\draw[->] (1,0.3) -- (1,.8) node[pos=0.5,right] {$d^1$};
 \draw[->] (2,0.3) -- (2,.8);
 \draw[->] (2,1.3) -- (2,1.8);
\end{tikzpicture}
\end{equation*}

The only way for this spectral sequence to exhibit the expected convergence is to have the differential $d^1_{1,0}: E^1_{1,0} \to E^1_{1,1}$ satisfy
\begin{equation*}
\mathrm{im} \ \! d^1_{1,0} \cong \mb Z^{n-1} \text{ and } \ker d^1_{1,0} \cong 0.
\end{equation*}

This implies $E^1_{1,0}\cong \tilde H_0(\I {K_n})\cong \mb Z^{n-1}$ and $E^1_{p,0}\cong \tilde H_{p-1}(\I {K_n})\cong 0$ for $p\neq 1$.
\end{eg}

In the general case there is a similar behavior with respect to maximum (and maximal) independent sets. 

\begin{prop}\label{thm:properties}
Let $\alpha(G)$ denote the \textit{independence number} of $G$, i.e.\ the cardinality of a maximum independent set in $G$. Then we have for the spectral sequence from Corollary \ref{cor:specseq}:
\begin{enumerate}
\item $E^\infty_{p,q} \cong \begin{cases}
                 \mb Z & p=q= \alpha(G) \\
                 0 & \text{ else}
                \end{cases}$
\item $E^1_{p,p}\cong \mb Z^{n_p}$ where $n_p$ is the number of maximal independent sets of size $p$ in $G$
\item all diagonal entries with $p<\alpha(G)$ vanish already on the $E^2$ page, 
\begin{equation*}
E^2_{p,p}=H_p(E^1_{p,\bullet},d^1) \cong H_p(H_p(T_{\bullet,\bullet},d),\delta)\cong0
\end{equation*}
\end{enumerate}
\end{prop}

\begin{proof}
Assertion (1) holds by construction because the spectral sequence converges to the associated graded piece of the homology of $(T,D)$. It follows however also from (3) by Theorem \ref{thm:acyclic} which assures that the surviving entry must lie on the diagonal; these entries represent the total degree 0 in which the only nontrivial homology of $(T,D)$ is concentrated. 

To prove (2) we consider the diagonal entries $E^1_{p,p}=H_p(T_{\bullet,p},d)$. Proposition \ref{prop:dcomplex} shows that this group is nonzero if and only if $G$ has a maximal independent set of cardinality $p$ (because every nonempty graph has $H_0(T_{\bullet,0},d) \cong 0$). Since $d$ does not act on 2-marked vertices, the complex $(T_{\bullet,p},d)$ splits into a direct sum of complexes, one for each 2-marking of size $p$. Thus, we find one generator for each maximal independent set of cardinality $p$.
In particular, if $G$ does not have such maximal independent sets, then $E^1_{p,p}\cong 0$.

For (3) let now $p<\alpha(G)$. The map $d^1$ on $E^1$ is given by the restriction of $\delta$ to the homology classes of $(T,d)$. Our goal is therefore to find for each generator of $H_p(T_{\bullet,p},d)$ a homology class in $H_p(T_{\bullet,p-1},d)$ that gets mapped to it by $\delta$. Note that, if it were not for the restriction to homology classes of $d$, this task would be trivial as the homology with respect to $\delta$ vanishes (Proposition \ref{prop:deltahomtrivial}).

By (2) the generators of $H_p(T_{\bullet,p},d)$ are represented by maximal independents sets $I$ with $|I|=p$. For each such $I$ there exists a vertex $z\in I$ such that the graph $G\mi N[I\setdiff \{z\}]$ has at least two vertices that can be simultaneously marked -- otherwise $I$ would be not only maximal, but also maximum independent. 

Choose two such vertices, denote them by $x$ and $y$, and consider the markings
\begin{equation*}
 m_0: v \mapsto \begin{cases}
               2 & \text{ if } v \in I, \\ 0 & \text{ else,}
              \end{cases}
             \quad
 m_z: v \mapsto \begin{cases}
               1 & \text{ if }v=z, \\ 2 & \text{ if } v \in I\setdiff\{ z\}, \\ 0 & \text{ else,}
              \end{cases}
            \end{equation*}
\begin{equation*}
            m_x: v \mapsto \begin{cases}
               1 & \text{ if }v=x, \\ 2 & \text{ if } v \in I\setdiff \{x,z\}, \\ 0 & \text{ else,}
              \end{cases}
              \quad  
              m_{x,y}: v \mapsto \begin{cases}
               1 & \text{ if }v =y, \\ 2 & \text{ if } v \in I\setdiff \{y,z\}, \\ 0 & \text{ else.}
              \end{cases}
\end{equation*}
See Figure \ref{fig:zeroclass} for an example.

We have $d(m_z - m_x)=0$ and, since $\{v,z\} \in E$ holds for every vertex $v$ of $G\mi N[I\setdiff \{z\}]$, the element $m_z - m_x$ cannot lie in the image of $d$. Hence, 
\begin{equation*}
 0 \neq [m_z - m_x] \in H_p(T_{\bullet,p-1},d).
\end{equation*}

Moreover, since $\delta m_x= d m_{x,y}$, we have found a $\delta$-preimage of $[m_0]$, 
\begin{equation*}
 [\delta (m_z-m_x)] = [m_0- \delta m_x] = [m_0] \in H_p(T_{\bullet,p},d).
\end{equation*} 
\end{proof}

\begin{figure}
 \begin{tikzpicture}[scale=0.5]
  \coordinate  (v1) at (-2,-2); 
   \coordinate  (v2) at (2,-2);
   \coordinate (v3) at (2,2); 
    \coordinate (v4) at (-2,2) ; 
   \draw (v1) -- (v2);
   \draw (v2) -- (v3);
   \draw (v3) -- (v4);
   \draw (v4) -- (v1);
   \coordinate  (w1) at (-1,-1); 
   \coordinate  (w2) at (1,-1);
   \coordinate (w3) at (1,1); 
    \coordinate (w4) at (-1,1); 
    \draw (w1) -- (w2);
   \draw (w2) -- (w3);
   \draw (w3) -- (w4);
   \draw (w4) -- (w1);
   \draw (v1) -- (w1);
   \draw (v2) -- (w2);
   \draw (v3) -- (w3);
   \draw (v4) -- (w4);
  \filldraw[fill=black] (v1)node[above left,cyan]{$y$} node[xshift=-.5cm,yshift=.9cm]{$m_0=$}circle (.0666) ;
  \filldraw[fill=black] (v2) circle (.0666) ;
  \filldraw[fill=black] (v3)node[right,cyan]{$x$} circle (.0666) ;
  \filldraw[fill=white] (v4)node[left,cyan]{$z$} circle (0.13);
  \filldraw[fill=black] (w1) circle (.0666) ;
  \filldraw[fill=white] (w2) circle (.13);
  \filldraw[fill=black] (w3) circle (.0666) ;
  \filldraw[fill=black] (w4) circle (.0666) ;
   \end{tikzpicture} 
  \begin{tikzpicture}[scale=0.5]
  \coordinate  (v1) at (-2,-2); 
   \coordinate  (v2) at (2,-2);
   \coordinate (v3) at (2,2); 
    \coordinate (v4) at (-2,2) ; 
   \draw (v1) -- (v2);
   \draw (v2) -- (v3);
   \draw (v3) -- (v4);
   \draw (v4) -- (v1);
   \coordinate  (w1) at (-1,-1); 
   \coordinate  (w2) at (1,-1);
   \coordinate (w3) at (1,1); 
    \coordinate (w4) at (-1,1); 
    \draw (w1) -- (w2);
   \draw (w2) -- (w3);
   \draw (w3) -- (w4);
   \draw (w4) -- (w1);
   \draw (v1) -- (w1);
   \draw (v2) -- (w2);
   \draw (v3) -- (w3);
   \draw (v4) -- (w4);
  \filldraw[fill=black] (v1)node[xshift=-.5cm,yshift=.9cm]{$m_z=$} circle (.0666) ;
  \filldraw[fill=black] (v2) circle (.0666) ;
  \filldraw[fill=black] (v3) circle (.0666) ;
  \filldraw[fill=orange] (v4) circle (0.13);
  \filldraw[fill=black] (w1) circle (.0666) ;
  \filldraw[fill=white] (w2) circle (.13);
  \filldraw[fill=black] (w3) circle (.0666) ;
  \filldraw[fill=black] (w4) circle (.0666) ;
   \end{tikzpicture} 
    \begin{tikzpicture}[scale=0.5]
  \coordinate  (v1) at (-2,-2); 
   \coordinate  (v2) at (2,-2);
   \coordinate (v3) at (2,2); 
    \coordinate (v4) at (-2,2) ; 
   \draw (v1) -- (v2);
   \draw (v2) -- (v3);
   \draw (v3) -- (v4);
   \draw (v4) -- (v1);
   \coordinate  (w1) at (-1,-1); 
   \coordinate  (w2) at (1,-1);
   \coordinate (w3) at (1,1); 
    \coordinate (w4) at (-1,1); 
    \draw (w1) -- (w2);
   \draw (w2) -- (w3);
   \draw (w3) -- (w4);
   \draw (w4) -- (w1);
   \draw (v1) -- (w1);
   \draw (v2) -- (w2);
   \draw (v3) -- (w3);
   \draw (v4) -- (w4);
  \filldraw[fill=black] (v1)node[xshift=-.5cm,yshift=.9cm]{$m_x=$} circle (.0666) ;
  \filldraw[fill=black] (v2) circle (.0666) ;
  \filldraw[fill=black] (v4) circle (.0666) ;
  \filldraw[fill=orange] (v3) circle (0.13);
  \filldraw[fill=black] (w1) circle (.0666) ;
  \filldraw[fill=white] (w2) circle (.13);
  \filldraw[fill=black] (w3) circle (.0666) ;
  \filldraw[fill=black] (w4) circle (.0666) ;
   \end{tikzpicture} 
       \begin{tikzpicture}[scale=0.5]
  \coordinate  (v1) at (-2,-2); 
   \coordinate  (v2) at (2,-2);
   \coordinate (v3) at (2,2); 
    \coordinate (v4) at (-2,2) ; 
   \draw (v1) -- (v2);
   \draw (v2) -- (v3);
   \draw (v3) -- (v4);
   \draw (v4) -- (v1);
   \coordinate  (w1) at (-1,-1); 
   \coordinate  (w2) at (1,-1);
   \coordinate (w3) at (1,1); 
    \coordinate (w4) at (-1,1); 
    \draw (w1) -- (w2);
   \draw (w2) -- (w3);
   \draw (w3) -- (w4);
   \draw (w4) -- (w1);
   \draw (v1) -- (w1);
   \draw (v2) -- (w2);
   \draw (v3) -- (w3);
   \draw (v4) -- (w4);
  \filldraw[fill=orange] (v1)node[xshift=-.6cm,yshift=.9cm]{$m_{x,y}=$} circle (.13) ;
  \filldraw[fill=black] (v2) circle (.0666) ;
  \filldraw[fill=black] (v4) circle (.0666) ;
  \filldraw[fill=white] (v3) circle (0.13);
  \filldraw[fill=black] (w1) circle (.0666) ;
  \filldraw[fill=white] (w2) circle (.13);
  \filldraw[fill=black] (w3) circle (.0666) ;
  \filldraw[fill=black] (w4) circle (.0666) ;
   \end{tikzpicture}
   \caption{Examples for the markings constructed in the proof of Proposition \ref{thm:properties}. Vertices marked by 1 are colored in orange, 2-marked vertices in white.}
	\label{fig:zeroclass}
\end{figure}
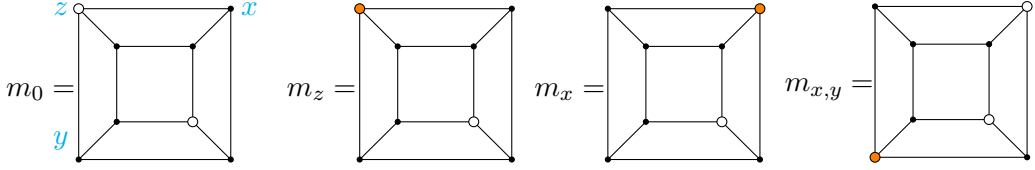

The preceding proposition improves the computational capability of the spectral sequence from Corollary \ref{cor:specseq}. This implies that in ``good'' cases the spectral sequence contains already enough information to fully determine the homology of $\mathrm{Ind}(G)$ (or at least to find relations between the groups in different dimensions). In ``not so good'' cases one has to examine the differential $d^1$ or throw in some additional information. Fortunately, $d^1$ is induced by the map $\delta$ and therefore rather simple.

Here the term ``good'' essentially means that we know or are able to compute the homology of the independence complexes of the graphs $G\mi N[U]$ for $U \subset V$ independent. For instance, if $G$ has many vertices of high valence or is a very symmetric graph, then the graphs $G\mi N[U]$ become very simple as the cardinality of $U$ grows. This is demonstrated by the examples in the next section.  
\newline

Last, but not least, there is one peculiar property of the spectral sequence that holds for a large class of examples, including all paths and cyclic graphs. 

\begin{thm} \label{thm:vp}
Let $G$ be a path or cyclic graph.
If $E^1_{p,p}\cong 0$, i.e.\ $G$ has no maximal independent set of cardinality $p$, then all entries of the upper column $E^1_{p,q}$, $q>0$, vanish.
\end{thm}

This theorem has strong implications. Together with
\begin{equation*}
 E^1_{p,q} = H_p(T_{\bullet,q},d) \overset{\eqref{eq:dcomplex2}}{\cong} \bigoplus_{ \substack{ U\subset V \text{ independent}  \\ |U|=q }  } H_{p-q}(T_{\bullet,0}(G\mi N[U]),d)
\end{equation*}
it allows to deduce the vanishing of the rank $p-q-1$ homology groups of the independence complex of \emph{every} subgraph of $G$ that can be obtained by deleting $q$ independent vertices and their neighborhoods.\footnote{In a way reflecting the ancient calculus tables where solutions of integrals are listed that were found by differentiating all kinds of functions.}

\begin{proof}
Throughout this proof we abbreviate $H_i(T_{\bullet,0},d)$ by $H_i$ and we drop set brackets in the notation for induced subgraphs.

Let $G$ be a path or cyclic graph and $p>1$ such that $G$ has no maximal independent sets of size $p$. Let $I\subset V$ be independent with $\vert I \vert =p$. Note that it suffices to discuss the case of paths since $G\mi N[I]$ is a union of paths in both cases.

We will show for all $k\in \{2,\ldots, p\}$ and all $U\subset I$ with $ |U|=k$ that 
$$
H_{p-k}( G \mi N[U] ) \cong 0 
$$
implies
$$
H_{p-k+1}( G \mi N[U\setdiff v]) \cong 0
$$
for each $v \in U$.

To prove this we need three ingredients:
\begin{itemize}
\item A cofiber sequence, introduced in \cite{ADAMASZEK2012}, 
\begin{equation}\label{eq:cofiber}
\I {G\mi N[v]} \hookrightarrow \I {G \mi \{v\}} \hookrightarrow \I G \to
\Sigma \I {G\mi N[v]} \to \ldots,
\end{equation} 
where $\Sigma$ denotes unreduced suspension. It expresses $\I G$ as the mapping cone of the subcomplex inclusion $\I {G\mi N[v]} \hookrightarrow \I {G\mi v}$. See \cite{ADAMASZEK2012} for the details.
\item The homotopy type of $\I {P_n}$, a path on $n$ vertices, shown in \cite{kozlov} to be
\begin{equation}\label{eq:homtype}
\I {P_n} \simeq 
\begin{cases}
S^{k-1} & \text{ if } n=3k \text{ or } n=3k-1,\\
\{\mathrm{pt}\} & \text{ else.} 
\end{cases}
\end{equation}
\item The fact that $\I {G \sqcup G'}= \I G * \I {G'}$ with $*$ denoting the topological join operation. In that regard, recall also that $S^i * S^j $ is homeomorphic to $S^{i+j+1}$.
\end{itemize} 

The sequence \eqref{eq:cofiber} gives rise to a long exact sequence in homology,\footnote{using that the mapping cone of a subcomplex inclusion $A \hookrightarrow B$ is $B$ with a cone attached over $A$ and thus homotopy equivalent to $B/A$.}
\begin{equation*}
\ldots \to \tilde{H}_n(\I {G\mi N[v]}) \to \tilde{H}_n(\I {G \mi v}) \to \tilde{H}_n( \I G ) \to
\tilde{H}_{n-1}(\I {G\mi N[v]} )\to \ldots
\end{equation*} 
Applying this to $G\mi N[U\setdiff v]=G \mi N[U] \pl N[v]$ with $U$ as above and $v\in U$ (here $+$ means ``putting a set of vertices back into the graph", $G \mi Y \pl X := G \mi (Y \setdiff X)$ for $X\subset Y\subset V$), the cofiber sequence reads
\begin{equation*}
\I {G\mi N[U]} \hookrightarrow \I {G\mi N[U] \pl N(v) } \hookrightarrow \I {G\mi N[U] \pl N[v]} =\I {G\mi N[U\setdiff v]}  \to \ldots
\end{equation*} 
The associated long exact sequence in homology is
\begin{equation} \label{eq:les}
\ldots \to H_n(G\mi N[U]) \to H_n( G\mi N[U] +N(v) ) \to H_n( G\mi N[U\setdiff v] ) \overset{\partial}{\to}
H_{n-1}( G\mi N[U] )\to \ldots
\end{equation} 
where we used the identification $H_i(T_{\bullet,0}(G),d) \cong \tilde H_{i-1}(\I G)$. 

Our goal is thus to show that for $n=p-k+1$ and $H_{n-1}( G\mi N[U] )\cong 0$ the map $\partial$ is an isomorphism.
\newline

1.\ Case $k=p$, $U=I$:

Let $v\in U$. By assumption, $H_0( G\mi N[U] )\cong 0$, so the graph $G\mi N[U]$ is nonempty. Putting $N[v]$ back into $G\mi N[U]$ ``creates'' at least two more vertices. 

If the resulting graph $G\mi N[U]+N[v]=G\mi N[U\setdiff v]$ has more than three vertices, then by \eqref{eq:homtype} its independence complex is homotopy equivalent to either a point or a sphere $S^m$ with $m>0$, hence $H_1( G\mi N[U\setdiff v] )\cong 0$. 

The case that $G\mi N[U\setdiff v]$ has exactly three vertices is possible only in one configuration; it is a path $P_3$ containing the first or last vertex of $G$ which must have been $v$. If we mark instead of $v$ its neighbor $v'$, then $G\mi N[(U\setdiff v) \cup v']=\emptyset$, a contradiction to the assumption that $G$ does not admit a maximal independent set of size $p$. 

See Figure \ref{fig:egcase1} for an example.

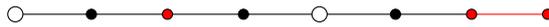
\begin{figure}[h]
\begin{tikzpicture}[scale=1]
  \coordinate  (v1) at (0,0); 
   \coordinate  (v2) at (1,0);
   \coordinate (v3) at (2,0); 
   \coordinate (v4) at (3,0); 
   \coordinate (v5) at (4,0); 
   \coordinate (v6) at (5,0); 
   \coordinate (v7) at (6,0); 
    \coordinate (v8) at (7,0); 
   \draw (v1) -- (v2) ;
   \draw (v2) -- (v3) ;
    \draw (v3) -- (v4) ;
     \draw (v4) -- (v5);
      \draw (v5) -- (v6) ;
       \draw (v6) -- (v7) ;
       \draw[red] (v7) -- (v8) ;
    \filldraw[fill=white] (v1) circle (.1);
  \filldraw[fill=black] (v2) circle (.0666) ;
  \filldraw[fill=red] (v3) circle (.0666);
    \filldraw[fill=black] (v4) circle (.0666);
      \filldraw[fill=white] (v5) circle (.1);
        \filldraw[fill=black] (v6) circle (.0666);
          \filldraw[fill=red] (v7) circle (.0666);
           \filldraw[fill=red] (v8) circle (.0666);
   \end{tikzpicture}
   \caption{An example for the case $k=p=2$ in the graph $P_8$. 2-marked vertices are colored white, the graph $G\mi N[U]$ is colored red. For a non-example remove the two rightmost vertices and keep the same marking.}
   \label{fig:egcase1}
\end{figure} 

2.\ Case $2\leq k < p$: 

Let $U \subset I$ with $|U|=k$ and $v \in U$. The relevant part of the long exact sequence \eqref{eq:les} is given by
$$
 H_{p-k+1}( G\mi N[U] ) \overset{i_*}{\to} 
H_{p-k+1}( G\mi N[U] +N(v) ) \overset{j_*}{\to} H_{p-k+1}( G\mi N[U\setdiff v] ) \overset{\partial}{\to}
H_{p-k}( G\mi N[U] ).
$$
Putting $N(v)$ back into the graph $ G\mi N[U]$ has one of the following effects:
\begin{enumerate}
 \item[(a)] It creates one or two isolated vertices in $G\mi N[U] \pl N(v)$.
 \item[(b)] $G\mi N[U]$ is a disjoint union of paths and $G\mi N[U] \pl N(v)$ is as well, but with one or two components lengthened by one vertex.
 \item[(c)] $G\mi N[U] = G\mi N[U] \pl N(v)$. This happens if and only if $v$ is the first or last vertex of $G$ and its neighbor's neighbor is already marked.
\end{enumerate}

All three cases have $H_{p-k+1}( G\mi N[U\setdiff v] )\cong 0$ as consequence, because
\begin{enumerate}
 \item[(a)] $G\mi N[U] \pl N(v) \simeq \{\mathrm{pt}\}$, since $\I {G' \sqcup v'}\simeq \I {G'}*\{ \mathrm{pt}\}$ is a cone on $\I {G'}$ for any isolated vertex $v'$ in any graph $G'$.
 
 \item[(b)] $G\mi N[U] = P_{n_1} \sqcup \ldots \sqcup P_{n_i}$ and $G\mi N[U] \pl N(v)$ is isomorphic to
\begin{equation*}
   \text{either }  P_{n_1+1} \sqcup P_{n_2} \sqcup \ldots \sqcup P_{n_i} \text{ or }  P_{n_1+1} \sqcup P_{n_2+1}\sqcup P_{n_3} \sqcup  \ldots \sqcup P_{n_i}.
\end{equation*}
From \eqref{eq:homtype} we see that $G\mi N[U] \pl N(v)$ is homotopy equivalent to a point if $n_1$ and $n_2$ are both not congruent 2 modulo 3. On the other hand, if they are, then 
\begin{equation*}
 \I {G\mi N[U]} \simeq \I {G\mi N[U]\pl N(v)}, 
\end{equation*}
so
\begin{equation*}
 H_\bullet(G\mi N[U]) \cong H_\bullet(G\mi N[U]\pl N(v)).
\end{equation*}
This means $i_*$ is an isomorphism. Since \eqref{eq:les} is exact, we get $\ker{\partial} \cong \image{j_*}\cong 0 $.
\item[(c)] $i_*$ is the identity map. Again, by exactness, $\ker{\partial} \cong  0 $.
\end{enumerate}

See Figure \ref{fig:egcase2} for an example.
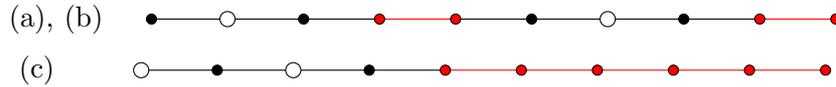
\begin{figure}[h]
\begin{tikzpicture}[scale=1]
  \coordinate  (v1) at (0,0); 
   \coordinate  (v2) at (1,0);
   \coordinate (v3) at (2,0); 
   \coordinate (v4) at (3,0); 
   \coordinate (v5) at (4,0); 
   \coordinate (v6) at (5,0); 
   \coordinate (v7) at (6,0); 
    \coordinate (v8) at (7,0); 
        \coordinate (v9) at (8,0); 
            \coordinate (v10) at (9,0); 
   \draw (v1) -- (v2) ;
   \draw (v2) -- (v3) ;
    \draw (v3) -- (v4) ;
     \draw[red] (v4) -- (v5);
      \draw (v5) -- (v6) ;
       \draw (v6) -- (v7) ;
       \draw (v7) -- (v8) ;
       \draw (v8) -- (v9) ;
       \draw[red] (v9) -- (v10) ;
    \filldraw[fill=black] (v1) circle (.0666) node[left, xshift=-.5cm] {(a), (b)};
  \filldraw[fill=white] (v2) circle (.1) ;
  \filldraw[fill=black] (v3) circle (.0666);
    \filldraw[fill=red] (v4) circle (.0666);
      \filldraw[fill=red] (v5) circle (.0666);
        \filldraw[fill=black] (v6) circle (.0666);
          \filldraw[fill=white] (v7) circle (.1);
           \filldraw[fill=black] (v8) circle (.0666);
             \filldraw[fill=red] (v9) circle (.0666);
               \filldraw[fill=red] (v10) circle (.0666);
   \end{tikzpicture}
   \begin{tikzpicture}[scale=1]
  \coordinate  (v1) at (0,0); 
   \coordinate  (v2) at (1,0);
   \coordinate (v3) at (2,0); 
   \coordinate (v4) at (3,0); 
   \coordinate (v5) at (4,0); 
   \coordinate (v6) at (5,0); 
   \coordinate (v7) at (6,0); 
    \coordinate (v8) at (7,0); 
        \coordinate (v9) at (8,0); 
            \coordinate (v10) at (9,0); 
   \draw (v1) -- (v2) ;
   \draw (v2) -- (v3) ;
    \draw (v3) -- (v4) ;
     \draw (v4) -- (v5);
      \draw[red] (v5) -- (v6) ;
       \draw[red] (v6) -- (v7) ;
       \draw[red] (v7) -- (v8) ;
       \draw[red] (v8) -- (v9) ;
       \draw[red] (v9) -- (v10) ;
    \filldraw[fill=white] (v1) circle (.1) node[left, xshift=-1cm] {(c)};
  \filldraw[fill=black] (v2) circle (.0666) ;
  \filldraw[fill=white] (v3) circle (.1);
    \filldraw[fill=black] (v4) circle (.0666);
      \filldraw[fill=red] (v5) circle (.0666);
        \filldraw[fill=red] (v6) circle (.0666);
          \filldraw[fill=red] (v7) circle (.0666);
           \filldraw[fill=red] (v8) circle (.0666);
             \filldraw[fill=red] (v9) circle (.0666);
               \filldraw[fill=red] (v10) circle (.0666);
   \end{tikzpicture}
   \caption{Examples for the cases (a), (b) and (c) with $k=2<p=3$ in the graph $P_{10}$. 2-marked vertices are colored in white, the graph $G\mi N[U]$ is colored in red.}
   \label{fig:egcase2}
\end{figure} 
\end{proof}

Although the preceding proof is specifically tailored to the case of paths and cyclic graphs, the statement holds also for many other graphs, including the examples in the next section, some cubic graphs and possibly all ladder graphs (checked for up to six rungs).

It is therefore natural to ask the following
\begin{q}
For which (families of) graphs $G$ does the statement of Theorem \ref{thm:vp} hold?
\end{q}

For instance, the independence complex of a forest is homotopy equivalent to either a point or a sphere (Corollary 6.1 in \cite{ehrenborg}, see also \cite{engstrom}), but a concrete characterization as in \ref{eq:homtype} is not known. Nevertheless, this suggests that Theorem \ref{thm:vp} could very well hold also for forests. 

An obvious further generalization is 
\begin{q}
Does the statement of Theorem \ref{thm:vp} hold for higher independence complexes? If yes, for which (families of) graphs?
\end{q}

\section{Examples}\label{sec:eg}
In this section we look at some examples to see the machinery at work.

Our first goal is to compute the homology of $\I P$ for $P$ the Petersen graph. For this we need a preparatory calculation. Throughout this section let $H_n$ denote $H_n(T_{\bullet,0},d)$.

\begin{eg}[$C_6$, the cyclic graph on six vertices]
To set up the spectral sequence for $H(\I {C_6})$ we need to calculate the homology of the complexes in \eqref{eq:dcomplex2} for $i=1,2,3$.  
 
 Removing any vertex with its neighbors from $C_6$ gives a path $P_3$ on three vertices, so 
 \begin{equation*}
(T_{\bullet,1},d) \cong \bigoplus_{k=1}^6 (T_{\bullet-1}(P_3),d).
\end{equation*}
 The homology of $\I{P_3}$ may be calculated directly, or simply read off from Figure \ref{fig:g,icg,2icg}. We find
 \begin{equation*}
  H_n(T_{\bullet,1},d) \cong \begin{cases} \mb Z^6 & n=1, \\ 0 & n \neq 1. \end{cases}
 \end{equation*}

 Removing two independent vertices together with their neighborhoods produces either an empty graph or a single vertex. The latter has trivial homology while the former adds a copy of $\mb Z$ in degree zero. There are three different maximum independent sets of size two, so 
  \begin{equation*}
 H_n(T_{\bullet,2},d) \cong \begin{cases} \mb Z^3 & n=0, \\ 0 & n \neq 0. \end{cases}
 \end{equation*}

Finally, $C_6$ has two maximal independent sets of size three, 
 \begin{equation*}
 H_n(T_{\bullet,3},d) \cong \begin{cases} \mb Z^2 & n=0, \\ 0 & n \neq 0. \end{cases}
 \end{equation*}

Filling out the first page of the associated spectral sequence
\vspace{.35cm}
\begin{equation*}
\begin{sseq}[grid=none,labels=none,entrysize=.8cm]{0...3}{0...3}
\ssdrop{0} \ssmove{1}{0} \ssdrop{H_1} \ssmove{1}{0} \ssdrop{H_2} \ssmove{1}{0} \ssdrop{H_3} 

\ssmoveto{1}{1} \ssdrop{0} \ssmove{1}{0} \ssdrop{\mb Z^6} \ssmove{1}{0} \ssdrop{0} 

\ssmoveto{2}{2} \ssdrop{\mb Z^3} \ssmove{1}{0} \ssdrop{0} 

\ssmoveto{3}{3} \ssdrop{\mb Z^2} 
\end{sseq}
\end{equation*}
we immediately deduce that $H_3$ must vanish and $H_2 \cong \mathrm{im}\ \! d^1_{2,0} \leq \mb Z^6$. The next page is
\vspace{.35cm}
\begin{equation*}
\begin{sseq}[grid=none,labels=none,entrysize=.8cm]{0...3}{0...3}
\ssdrop{0} \ssmove{1}{0} \ssdrop{H_1} \ssmove{1}{0} \ssdrop{0} \ssmove{1}{0} \ssdrop{0} 

\ssmoveto{1}{1} \ssdrop{0} \ssmove{1}{0} \ssdrop{E^2_{2,1} } \ssmove{1}{0} \ssdrop{0} 

\ssmoveto{2}{2} \ssdrop{E^2_{2,2}} \ssmove{1}{0} \ssdrop{0} 

\ssmoveto{3}{3} \ssdrop{\mb Z^2} 
\end{sseq}
\end{equation*}
with $E^2_{2,1} \cong \ker d^1_{2,1} / H_2$ and $E^2_{2,2}  \cong \mb Z^3 / \mathrm{im}\ \! d^1_{2,1}$. From Proposition \ref{thm:properties} we know that $E^2_{2,2}\cong 0$. Hence, $\image d^1_{2,1} \cong \mb Z^3$ and $\ker d^1_{2,1} \cong \mb Z^3$.

Since this page's differential $d^2$ goes one step to the right and two steps up, we must have $E^2_{2,1} \cong \mb Z$ and $H_1 \cong E^2_{2,2}\cong 0$ for the spectral sequence to exhibit its expected convergence behavior.
Putting everything together we conclude
\begin{equation*}
 H_1 \cong 0, \ H_2 \cong \mb Z^2 \ \Longrightarrow \tilde{H}_n(\I {C_6}) \cong \begin{cases}
                                                                                  \mb Z^2 & \text{ if } n=1, \\
                                                                                  0 & \text{ else.}
                                                                                 \end{cases}
\end{equation*}
\end{eg}

\begin{eg}[The Petersen graph $P$]
 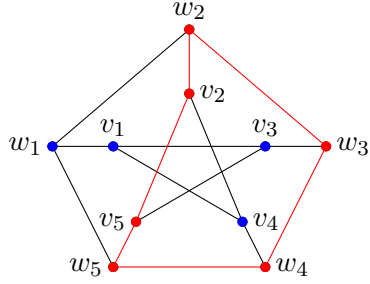
\begin{figure}
  \begin{tikzpicture}[scale=1]
  \coordinate  (v1) at (-1,0); 
   \coordinate  (v2) at (0,0.7);
   \coordinate (v3) at (1,0);
   \coordinate (v4) at (0.7,-1);
   \coordinate (v5) at (-0.7,-1);
   \draw (v1) -- (v3) ;
   \draw (v1) -- (v4) ;
   \draw (v2) -- (v4) ;
   \draw[red] (v2) -- (v5) ;
   \draw (v3) -- (v5) ;
  \fill[] (v1) circle (.0666cm) node[above]{$v_1$};
   \fill[] (v2) circle (.0666cm) node[right]{$v_2$};
\fill[] (v3) circle (.0666cm) node[above]{$v_3$};
\fill[] (v4) circle (.0666cm) node[right]{$v_4$};
\fill[] (v5) circle (.0666cm) node[left]{$v_5$};
  \coordinate  (w1) at (-1.8,0); 
   \coordinate  (w2) at (0,1.55);
   \coordinate (w3) at (1.8,0);
   \coordinate (w4) at (1,-1.6);
   \coordinate (w5) at (-1,-1.6);
   \draw (v1) -- (w1) ;
   \draw (v3) -- (w3) ;
   \draw (v4) -- (w4) ;
    \draw (w1) -- (w2) ;
   \draw (w5) -- (w1) ;
  \fill[] (w1) circle (.0666cm) node[left]{$w_1$};
   \fill[] (w2) circle (.0666cm) node[above]{$w_2$};
\fill[] (w3) circle (.0666cm) node[right]{$w_3$};
\fill[] (w4) circle (.0666cm) node[right]{$w_4$};
\fill[] (w5) circle (.0666cm) node[left]{$w_5$}; 
\draw[red] (w2) -- (w3) ;
   \draw[red] (w3) -- (w4) ;
   \draw[red] (w4) -- (w5) ;
   \draw[red] (v5) -- (w5) ;
   \draw[red] (v2) -- (w2) ;
 \fill[blue] (v1) circle (.0666cm);
 \fill[blue] (w1) circle (.0666cm);
 \fill[blue] (v3) circle (.0666cm);
 \fill[blue] (v4) circle (.0666cm);
\fill[red] (w2) circle (.0666cm);
\fill[red] (v2) circle (.0666cm);
\fill[red] (v5) circle (.0666cm);
\fill[red] (w5) circle (.0666cm);
\fill[red] (w4) circle (.0666cm);
\fill[red] (w3) circle (.0666cm);
\end{tikzpicture} 
   \caption{The Petersen graph $P$. The closed neighborhood $N[v_1]$ is depicted in blue, its complement graph $P\mi N[v_1]$ in red.}\label{fig:peter}
 \end{figure}
  Removing the ball around each of the vertices of $P$ produces a cyclic graph $C_6$ on six vertices. See Figure \ref{fig:peter} for the case of an ``interior'' vertex and note that we get an isomorphic graph if we do the same with one of the $w_i$, $i=1, \ldots, 5$. 
  
  Using the previous example we have thus
    \begin{equation*}
 H_n(T_{\bullet,1},d) \cong \begin{cases} \mb Z^{20} & n=2, \\ 0 & n \neq 2. \end{cases}
 \end{equation*}

 Deleting another vertex and its neighbors in the remaining graph produces $P_3$, a path on three vertices. There are thirty different ways of doing so, hence
   \begin{equation*}
 H_n(T_{\bullet,2},d) \cong \begin{cases} \mb Z^{30} & n=1, \\ 0 & n \neq 1. \end{cases}
 \end{equation*}
 
 Arguing for the remaining $P_3$ as in the previous example and counting all different ways of deleting three vertices and their neighborhoods in $P$, we find
   \begin{equation*}
 H_n(T_{\bullet,3},d) \cong \begin{cases} \mb Z^{10} & n=0, \\ 0 & n \neq 0.\end{cases}
 \end{equation*}
 
 The Petersen graph has five maximum independent sets of size four, so that
 \begin{equation*}
 H_n(T_{\bullet,4},d) \cong \begin{cases} \mb Z^{5} & n=0, \\ 0 & n \neq 0. \end{cases}
 \end{equation*}
 
 The first page of the associated spectral sequence is then given by
 \vspace{.35cm}
\begin{equation*}
\begin{sseq}[grid=none,labels=none,entrysize=.8cm]{0...4}{0...4}
\ssdrop{0} \ssmove{1}{0} \ssdrop{H_1} \ssmove{1}{0} \ssdrop{H_2} \ssmove{1}{0} \ssdrop{H_3} \ssmove{1}{0} \ssdrop{H_4} 

\ssmoveto{1}{1} \ssdrop{0} \ssmove{1}{0} \ssdrop{0} \ssmove{1}{0}  \ssdrop{\mb Z^{20} } \ssmove{1}{0} \ssdrop{0} 

\ssmoveto{2}{2} \ssdrop{0} \ssmove{1}{0} \ssdrop{\mb Z^{30}} \ssmove{1}{0} \ssdrop{0} 

\ssmoveto{3}{3} \ssdrop{\mb Z^{10}}  \ssmove{1}{0}  \ssdrop{0}

\ssmoveto{4}{4} \ssdrop{\mb Z^{5}} 
\end{sseq}
\end{equation*}

We deduce $H_3 \cong \image d^1_{3,0}$ and $H_4 \cong 0$. Furthermore, it must hold that $\image d^1_{3,3}\cong \mb Z^{10}$ because $E^2_{3,3}\cong 0$ by Proposition \ref{thm:properties}. This implies $\ker d^1_{3,3}\cong \mb Z^{20}$.

Therefore, we find on the $E_2$-page
\begin{equation*}
\begin{sseq}[grid=none,labels=none,entrysize=.8cm]{0...4}{0...4}
\ssdrop{0} \ssmove{1}{0} \ssdrop{H_1} \ssmove{1}{0} \ssdrop{H_2} \ssmove{1}{0} \ssdrop{0} \ssmove{1}{0} \ssdrop{0} 

\ssmoveto{1}{1} \ssdrop{0} \ssmove{1}{0} \ssdrop{0} \ssmove{1}{0}  \ssdrop{X/H_3 } \ssmove{1}{0} \ssdrop{0} 

\ssmoveto{2}{2} \ssdrop{0} \ssmove{1}{0} \ssdrop{\mb Z^{20}/Y} \ssmove{1}{0} \ssdrop{0} 

\ssmoveto{3}{3} \ssdrop{0}  \ssmove{1}{0}  \ssdrop{0}

\ssmoveto{4}{4} \ssdrop{\mb Z^{5}} 
\end{sseq}
\end{equation*}
for $X:= \ker d^1_{3,1}$ and $Y:=\image d^1_{3,1}$, satisfying $X \oplus Y \cong \mb Z^{20}$. The only nontrivial differentials are
\begin{equation*}
 d^2_{1,0}:H_1 \to 0,\quad  d^2_{2,0}: H_2 \to \mb Z^{20}/Y,  \quad  d^2_{3,2}: \mb Z^{20}/Y \to  \mb Z^5, \quad d^2_{3,1}: X/H_3 \to 0.
\end{equation*}

Convergence of the spectral sequence implies
\begin{equation*}
 H_1 \cong 0,\quad   \image d^2_{3,2} \cong \mb Z^4, \quad \ker d^2_{3,2} \cong H_2, \quad X \cong H_3, 
\end{equation*}
and, using $X \oplus Y \cong \mb Z^{20}$, this is equivalent to $H_3 \cong \mb Z^4 \oplus H_2$. 

We see that the spectral sequence does not always solve the problem completely; additional calculations may be necessary. In the present case one finds $H_2\cong 0$, so that 
\begin{equation*}
 \tilde H_n(\I P) \cong \begin{cases} \mb Z^{4} & n=2, \\ 0 & n \neq 2. \end{cases}
 \end{equation*}
\end{eg}

\begin{eg}[$K$, the 1-skeleton of a three dimensional cube] Here the $E^1$-page of the associated spectral sequence looks like
 \begin{equation*}
\begin{sseq}[grid=none,labels=none,entrysize=.8cm]{0...4}{0...4}
\ssdrop{0} \ssmove{1}{0} \ssdrop{H_1} \ssmove{1}{0} \ssdrop{H_2} \ssmove{1}{0} \ssdrop{H_3} \ssmove{1}{0} \ssdrop{H_4} 

\ssmoveto{1}{1} \ssdrop{0} \ssmove{1}{0} \ssdrop{\mb Z^8} \ssmove{1}{0}  \ssdrop{0 } \ssmove{1}{0} \ssdrop{0} 

\ssmoveto{2}{2} \ssdrop{\mb Z^4} \ssmove{1}{0} \ssdrop{0} \ssmove{1}{0} \ssdrop{0} 

\ssmoveto{3}{3} \ssdrop{0}  \ssmove{1}{0}  \ssdrop{0}

\ssmoveto{4}{4} \ssdrop{\mb Z^{2}} 
\end{sseq}
\end{equation*}
from which it follows that
\begin{equation*}
 \tilde H_n(\I K) \cong \begin{cases} \mb Z^{3} & n=1, \\ 0 & n \neq 1. \end{cases}
 \end{equation*}
 
  The details of the computation are left to the reader. 
\end{eg}

We finish with two examples on the homology of 2-independence complexes. Note that Proposition \ref{prop:dcomplex} still applies, but in \eqref{eq:dcomplex} we have to replace the graphs $G- N[U]$ appropriately.

In the following let $T_{i,j}$ be given as in Definition \ref{defn:T}, except that markings $m:V \to \{0,1,2\}$ are now defined by requiring that $V_m=m^{-1}(\{1,2\})$ is a 2-independent set in $G$.

\begin{eg}[$C_4$, the cyclic graph on four vertices] \label{eg:2ind}
We consider the 2-independence complex $\mathrm{Ind}_2(C_4)$ whose geometric realization is homeomorphic to the 1-skeleton of a tetrahedron $\Delta^3$.

To fill out the first page of the associated spectral sequence we need to know the homology of the complexes $(T_{\bullet,j},d)$ for $j=1,2$.

Let a single vertex $v$ of $C_4$ be 2-marked. We may still mark any of the remaining three vertices without violating the condition of 2-independence, but not more. This means that in \eqref{eq:dcomplex} we have to replace each $G\mi N[v]$ by a $K_3$, the complete graph on the vertex set $V\setdiff \{ v\}$,
\begin{equation*}
(T_{\bullet,1},d) \cong \bigoplus_{k=1}^4 (C_{\bullet-1}(K_3),\partial).
\end{equation*}

Now let $j=2$, i.e.\ two vertices be 2-marked. Every such 2-independent set is already maximum, so
\begin{equation*}
(T_{\bullet,2},d) \cong \bigoplus_{ \substack{ U\subset V \text{ maximum} \\ \text{2-independent}}} (C_{\bullet-1}(\emptyset),\partial).
\end{equation*}

The homology groups of the latter two complexes are easy to compute: For the first one observe that $\I {K_3}$ consists of three disjoint vertices, for the second one note that there are six maximum 2-independent sets in $C_4$. Thus,
\begin{equation*}
H_n(T_{\bullet,1},d) \cong \begin{cases} (\mb Z^2)^4 & n=1, \\ 0 & n \neq 1, \end{cases} \text{ and } H_n(T_{\bullet,2},d) \cong \begin{cases} \mb Z^6 & n=0, \\ 0 & n \neq 0. \end{cases}
\end{equation*}

The first page of the associated spectral sequence is then
\vspace{.35cm}
\begin{equation*}
\begin{sseq}[grid=none,labels=none,entrysize=.8cm]{0...2}{0...2}
\ssdrop{0} \ssmove{1}{0} \ssdrop{H_1} \ssmove{1}{0} \ssdrop{H_2} 

\ssmoveto{1}{1} \ssdrop{0} \ssmove{1}{0} \ssdrop{\mb Z^8} 
\ssmoveto{2}{2} \ssdrop{\mb Z^6} 
\end{sseq}
\end{equation*}
which implies $H_1\cong 0$ and $H_2\cong \mb Z^3$ (the other possible solution, $H_1\cong \mb Z^5$ and $H_2\cong \mb Z^8$, cannot be true -- the rank of $H_1$ is always less than the number of vertices). We conclude 
\begin{equation*}\tilde{H}_n(\mathrm{Ind}_2(G))\cong \begin{cases}
                                                                                                  \mb Z^3 & \text{ if } n=1, \\
                                                                                                  0 & \text{ else.}                                                                                             
                                                                                                \end{cases}                                                                                                \end{equation*}
\end{eg}

\begin{eg}[$C_5$, the cyclic graph on five vertices] \label{eg:2ind2}
$C_5$ admits 2-independent sets of cardinality up to three, so we need to find the homology of the complexes $(T_{\bullet,j},d)$ for $j=1,2,3$.

Let a single vertex of $C_5$ be 2-marked, say $v_1$. Ordering the vertices of $C_5$ cyclically, its maximal 2-independent sets containing $v_1$ are 
\begin{equation*}
\{v_1,v_2,v_4 \}, \{v_1,v_3,v_4 \},\{v_1,v_3,v_5 \}.
\end{equation*}
To model all allowed 1-markings if $v_1$ is 2-marked, we have to replace in the decomposition formula \eqref{eq:dcomplex} the summand corresponding to  $C_5 - N[v_1]$ by a path $P_4$ on four vertices $v_2,v_3,v_4,v_5$ with edge set 
\begin{equation*}
E(P_4)= \big\{ \{v_3,v_2 \}, \{v_2,v_5 \}, \{v_5,v_4 \} \big\}.
\end{equation*}
Thus, using that $\I {P_4}$ is contractible,
\begin{equation*}
(T_{\bullet,1},d) \cong \bigoplus_{k=1}^5 (C_{\bullet-1}(P_4),\partial) \quad \Longrightarrow \quad H_n(T_{\bullet,1},d) \cong 0 \text{ for all } n \in \mb N.
\end{equation*}

Now let $j=2$, i.e.\ two vertices be 2-marked. The graphs encoding the remaining possible markings consist of either a single vertex or a $K_2$, two vertices connected by an edge (if we start with $v_1$ these cases correspond to 2-marking the sets $\{v_1,v_2\}$, $ \{v_1,v_5\}$ or $\{v_1,v_3\}$, $\{v_1,v_4\}$, respectively),
\begin{equation*}
(T_{\bullet,2},d) \cong \Big( \bigoplus_{ \{v_i,v_j\} \in  E} (C_{\bullet-1}(*),\partial) \Big) \oplus 
\Big( \bigoplus_{ \substack{ \{v_i,v_j\} \subset  V \\ |i-j| \in \{2,3\} }} (C_{\bullet-1}(K_2),\partial) \Big)
\end{equation*}
The first case has trivial homology, the latter contributes a copy of $\mb Z$ in degree one,
 \begin{equation*}
 H_n(T_{\bullet,2},d) \cong \begin{cases} \mb Z^5 & n=1, \\ 0 & n \neq 1. \end{cases}
\end{equation*}

Lastly, there are five maximal 2-independent sets of size three, 
\begin{equation*} 
H_n(T_{\bullet,3},d) \cong \begin{cases} \mb Z^5, & n=0, \\ 0 & n \neq 0. \end{cases}
\end{equation*}

Filling out the first page of the associated spectral sequence gives
\vspace{.35cm}
\begin{equation*}
\begin{sseq}[grid=none,labels=none,entrysize=.8cm]{0...3}{0...3}
\ssdrop{0} \ssmove{1}{0} \ssdrop{H_1} \ssmove{1}{0} \ssdrop{H_2} \ssmove{1}{0} \ssdrop{H_3} 

\ssmoveto{1}{1} \ssdrop{0} \ssmove{1}{0}  \ssdrop{0} \ssmove{1}{0}  \ssdrop{0} \ssmove{1}{0}

\ssmoveto{2}{2}  \ssdrop{0} \ssmove{1}{0} \ssdrop{\mb Z^5} 

\ssmoveto{3}{3}  \ssdrop{\mb Z^5} 
\end{sseq}
\end{equation*}
so that $H_3 \cong 0$. 
The next page $E^2$ reads 
\vspace{.35cm}
\begin{equation*}
\begin{sseq}[grid=none,labels=none,entrysize=.8cm]{0...3}{0...3}
\ssdrop{0} \ssmove{1}{0} \ssdrop{H_1} \ssmove{1}{0} \ssdrop{H_2} \ssmove{1}{0} \ssdrop{0} 

\ssmoveto{1}{1} \ssdrop{0} \ssmove{1}{0}  \ssdrop{0} \ssmove{1}{0}  \ssdrop{0} \ssmove{1}{0}

\ssmoveto{2}{2}  \ssdrop{0} \ssmove{1}{0} \ssdrop{X} 

\ssmoveto{3}{3}  \ssdrop{\mb Z^5/Y} 
\end{sseq}
\end{equation*}
with $X:=\ker d^1_{3,2}$ and $Y:= \image d^1_{3,2}$, $X\oplus Y \cong \mb Z^5$.

 For the spectral sequence to converge accordingly, we must have $H_2 \cong X \ncong 0$ (if $\ker d^1_{3,2}\cong 0$, then $Y\cong 0$) and $( \mb Z^5/Y ) / H_1 \cong \mb Z$. This shows $ H_2 \cong  H_1 \oplus\mb Z$. Now, either by inspecting the differential $d^1_{3,2}$ more closely, or by simply noting that $\mathrm{Ind}_2(C_5)$ is connected, we find $H_1 \cong 0$ and therefore
\begin{equation*}
 \tilde{H}_n(\mathrm{Ind}_2(C_5))\cong \begin{cases}\mb Z & n=1, \\ 0 & n\neq 1. \end{cases}
\end{equation*}

\end{eg}

\bibliography{ref}

\begin{thebibliography}{BMLN07}

\bibitem[Ada11]{adamaszek2}
Michal Adamaszek.
\newblock Special cycles in independence complexes and superfrustration in some
  lattices.
\newblock {\em Topology and its Applications}, 160, 09 2011.

\bibitem[Ada12]{ADAMASZEK2012}
Michał Adamaszek.
\newblock Splittings of independence complexes and the powers of cycles.
\newblock {\em Journal of Combinatorial Theory, Series A}, 119(5):1031 -- 1047,
  2012.

\bibitem[Ada15]{adamaszek}
Michal Adamaszek.
\newblock A note on independence complexes of chordal graphs and dismantling.
\newblock {\em Electronic Journal of Combinatorics}, 24, 08 2015.

\bibitem[Bar13]{barmak}
Jonathan~Ariel Barmak.
\newblock Star clusters in independence complexes of graphs.
\newblock {\em Advances in Mathematics}, 241:33 -- 57, 2013.

\bibitem[BK04]{babson-kozlov}
Eric Babson and Dmitry Kozlov.
\newblock Proof of the \uppercase{L}ovasz conjecture.
\newblock {\em Ann. of Math. (2)}, 165, 03 2004.

\bibitem[BK20]{mb-ak}
Marko Berghoff and Andre Knispel.
\newblock Complexes of marked graphs in gauge theory.
\newblock {\em Lett. Math. Phys.}, 2020.

\bibitem[BMLN07]{bm-sv-ne}
Mireille Bousquet-Mélou, Svante Linusson, and Eran Nevo.
\newblock On the independence complex of square grids.
\newblock {\em Journal of Algebraic Combinatorics}, 27, 03 2007.

\bibitem[Cho06]{chow-specseq}
Timothy~Y. Chow.
\newblock You could have invented spectral sequences.
\newblock {\em Notices Amer. Math. Soc.}, 53(1):15--19, 2006.

\bibitem[Cso09]{csorba}
Péter Csorba.
\newblock Subdivision yields \uppercase{A}lexander duality on independence
  complexes.
\newblock {\em Electr. J. Comb.}, 16, 05 2009.

\bibitem[DR17]{ramos-specseq}
Antonio D\'{\i}az~Ramos.
\newblock Spectral sequences via examples.
\newblock {\em Grad. J. Math.}, 2(1):10--28, 2017.

\bibitem[DS13]{dao-schweig}
Hailong Dao and Jay Schweig.
\newblock Projective dimension, graph domination parameters, and independence
  complex homology.
\newblock {\em Journal of Combinatorial Theory, Series A}, 120(2):453 -- 469,
  2013.

\bibitem[DS20]{deshp-singh}
Priyavrat Deshpande and Anurag Singh.
\newblock Higher independence complexes of graphs and their homotopy types,
  2020.
\newblock arXiv:2001.05448.

\bibitem[EH06]{ehrenborg}
Richard Ehrenborg and Gábor Hetyei.
\newblock The topology of the independence complex.
\newblock {\em European Journal of Combinatorics}, 27(6):906 -- 923, 2006.

\bibitem[Eng08]{engstrom2}
Alexander Engström.
\newblock Independence complexes of claw-free graphs.
\newblock {\em European Journal of Combinatorics}, 29(1):234 -- 241, 2008.

\bibitem[Eng09]{engstrom}
Alexander Engström.
\newblock Complexes of directed trees and independence complexes.
\newblock {\em Discrete Mathematics}, 309(10):3299 -- 3309, 2009.

\bibitem[GM99]{gm-homalg}
S.I. Gelfand and Y.I. Manin.
\newblock {\em Homological Algebra}.
\newblock Encyclopaedia of Mathematical Sciences. Springer, Berlin Heidelberg,
  1999.

\bibitem[HS09]{hui-schout}
Liza Huijse and K.~Schoutens.
\newblock Supersymmetry, lattice fermions, independence complexes and
  cohomology theory.
\newblock {\em Advances in Theoretical and Mathematical Physics}, 14, 03 2009.

\bibitem[Jon08]{jonsson}
J.~Jonsson.
\newblock {\em Simplicial Complexes of Graphs}.
\newblock Lecture Notes in Mathematics. Springer, Berlin Heidelberg, 2008.

\bibitem[Jon10]{jonsson2}
Jakob Jonsson.
\newblock Certain homology cycles of the independence complex of grids.
\newblock {\em Discrete \& Computational Geometry}, 43(4):927–950, 2010.

\bibitem[Kni17]{knispel}
A.~Knispel.
\newblock Combinatorial \uppercase{BRST} homology and graph differentials.
\newblock {\em Master's thesis}, 2017.
\newblock
  \url{http://www2.mathematik.hu-berlin.de/~kreimer/wp-content/uploads/Knispel.pdf}.

\bibitem[Koz99]{kozlov}
Dmitry~N. Kozlov.
\newblock Complexes of directed trees.
\newblock {\em Journal of Combinatorial Theory, Series A}, 88(1):112 -- 122,
  1999.

\bibitem[KSvS13]{ksvs}
D.~Kreimer, M.~Sars, and W.~van Suijlekom.
\newblock Quantization of gauge fields, graph polynomials and graph cohomology.
\newblock {\em Annals Phys.}, 336, 2013.

\bibitem[McC00]{ug-specseq}
John McCleary.
\newblock {\em A User's Guide to Spectral Sequences}.
\newblock Cambridge Studies in Advanced Mathematics. Cambridge University
  Press, 2 edition, 2000.

\bibitem[Mes03]{meshulam}
Roy Meshulam.
\newblock Domination numbers and homology.
\newblock {\em Journal of Combinatorial Theory, Series A}, 102(2):321 -- 330,
  2003.

\bibitem[PS17]{paolini-salvetti}
Giovanni Paolini and Mario Salvetti.
\newblock Weighted sheaves and homology of \uppercase{A}rtin groups.
\newblock {\em Algebraic \& Geometric Topology}, 18, 03 2017.

\bibitem[Sal15]{salvetti}
Mario Salvetti.
\newblock {\em Some Combinatorial Constructions and Relations with
  \uppercase{A}rtin Groups}, pages 121--126.
\newblock 01 2015.

\bibitem[ST06]{szabo-tardost}
Tibor Szabó and Gábor Tardos.
\newblock Extremal problems for transversals in graphs with bounded degree.
\newblock {\em Combinatorica}, 26(3):333–351, 2006.

\end{thebibliography}
\bibliographystyle{alpha}
\end{document}